\documentclass[reqno,11pt]{amsart}
\usepackage{mathtools,amsmath, amsfonts, amssymb, latexsym, amsthm,amscd, cancel, enumerate, rotating, comment, appendix, ulem,makecell, paralist, times, graphicx, pgf, tikz, tikz-cd, geometry, dirtytalk, enumitem}

\usepackage{tkz-euclide}
\usepackage{mathdots}
\usepackage{float}
\usepackage[capitalise, noabbrev]{cleveref}
\usepackage[latin1]{inputenc}
\usepackage[english]{babel}
\usepackage[mathscr]{eucal}
\usepackage{xxcolor}
\usepackage{caption}
\usepackage{subcaption}

\usepackage[T1]{fontenc}

\input{xy}
\xyoption{all}

\usepackage[all]{xy}
\usetikzlibrary{shapes}
\usetikzlibrary{snakes}
\usetikzlibrary{arrows}
\usetikzlibrary{matrix, decorations.pathmorphing}
\usetikzlibrary{decorations.markings}
\usetikzlibrary{calc}

\numberwithin{equation}{section}
\newif\ifdviwin

\newtheorem{theorem}{Theorem}[section]
\newtheorem{lemma}[theorem]{Lemma}

\newtheorem{proposition}[theorem]{Proposition}

\newtheorem*{theorem*}{Theorem}

\theoremstyle{definition}
\newtheorem{definition}[theorem]{{Definition}}
\newtheorem{example}[theorem]{Example}
\newtheorem{remark}[theorem]{Remark}

\newtheorem{Question}[theorem]{Question}

\theoremstyle{remark}

\newtheorem*{chunk*}{}

\newcommand{\st}{\colon}
\newcommand{\sm}{\setminus}

\newcommand{\Taylor}{\mathrm{Taylor}}
\newcommand{\Scarf}{\mathrm{Scarf}}

\newcommand{\im}{\mathrm{im}}
\newcommand{\m}{\mathbf{m}}
\newcommand{\bm}{\mathbf{m}}

\newcommand{\lcm}{\mathrm{lcm}}
\newcommand{\LCM}{\mathrm{LCM}}

\newcommand{\dist}{\mathrm{dist}}
\newcommand{\Gens}{\mathrm{Gens}}

\newcommand{\NN}{\mathbb{N}}

\newcommand{\qwhere}{\quad \mbox{ where } \quad }
\newcommand{\qand}{\quad \mbox{ and } \quad }

\newcommand{\qor}{\quad \mbox{ or } \quad }
\newcommand{\qfor}{\quad \mbox{ for } \quad }

\newcommand{\sfk}{\mathsf k}

\newcommand{\void}[1]{}

\begin{document}
\bibliographystyle{amsplain}

\author[S.~Faridi]{Sara Faridi}
\address{Department of Mathematics \& Statistics\\
Dalhousie University\\
6316 Coburg Rd.\\
PO BOX 15000\\
Halifax, NS\\
Canada B3H 4R2}
\email{faridi@dal.ca}

\author[T.~H.~H\`a]{T\`ai Huy H\`a}
\address{Tulane University\\
Department of Mathematics\\
6823 St. Charles Avenue\\
New Orleans, LA 70118, USA}
\email{tha@tulane.edu}

\author[T.~Hibi]{Takayuki Hibi}
\address{Department of Pure and Applied Mathematics\\
Graduate School
of Information Science and Technology \\
Osaka University, Suita \\
Osaka 565-0871, Japan}
\email{hibi@math.sci.osaka-u.ac.jp}

\author[S.~Morey]{Susan Morey}
\address{Department of Mathematics\\
Texas State University\\
601 University Dr.\\
San Marcos, TX 78666
\\USA}
\email{morey@txstate.edu}



\title{Scarf Complexes of Graphs and their Powers}

\begin{abstract} Every multigraded free resolution of a monomial ideal $I$ contains the Scarf multidegrees of $I$. We say $I$ has a Scarf resolution if the Scarf multidegrees are sufficient to describe a minimal free resolution of $I$.
The main question of this paper is which graphs $G$ have edge ideal $I(G)$ with a Scarf resolution?
 We show that $I(G)$ has a Scarf resolution  if and only if $G$ is a gap-free
 forest. We also classify connected graphs for which $I(G)^t$ has a Scarf
 resolution, for $t \geq 2$. Along the way, we give a concrete description of the Scarf complex of any
 forest. For a general graph, we give a recursive construction for its
 Scarf complex based on Scarf complexes of induced subgraphs.
\end{abstract}

\maketitle

\section{Introduction}
\label{s:introduction}
Constructing the minimal free resolution for a monomial ideal in a polynomial ring is a classical research topic in commutative algebra which continues to inspire current work. In essence, a minimal free resolution encodes dependence relations between polynomials. As a basic example, we can consider two single-term polynomials $f=xy$ and $g=yz$. Then the dependence relation between $f$ and $g$ is $zf-xg=0$. The (\say{multigraded}) minimal free resolution will keep track of these relations via an exact sequence of  free modules, indexed by the least common multiples of the variables appearing in each relation:
\begin{equation}\label{eq:res1} 0 \longrightarrow S(xyz) \longrightarrow S(xy)\oplus S(yz).\end{equation}
Here $S$ stands for the polynomial ring in three variables over a field. The minimal free resolution above can be thought of an exact sequence of maps of vector spaces, for all practical purposes. 

Our concern in this paper is the multidegrees that appear in the (minimal) free resolution of edge ideals of graphs, and of their powers.
The polynomials $f$ and $g$ above represent edges of the  graph $G_1$ below. 
$$
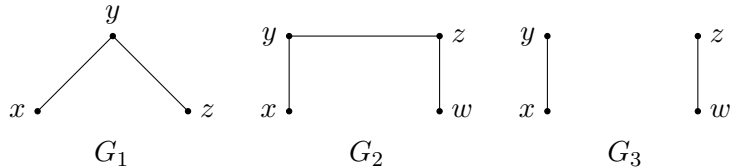
\begin{tabular}{ccc}
\begin{tikzpicture}
	\tikzstyle{point}=[inner sep=0pt]
	\node (a)[point, label=left:$x$] at (1,1) {};
	\node (b)[point, label=above:$y$] at (2,2) {};
	\node (c)[point, label=right:$z$] at (3,1) {};

	\draw (a.center) -- (b.center);
	\draw (c.center) -- (b.center);
	
	\filldraw [black] (a.center) circle (1pt);
	\filldraw [black] (b.center) circle (1pt);
	\filldraw [black] (c.center) circle (1pt);
	\end{tikzpicture}
 &
 \begin{tikzpicture}
	\tikzstyle{point}=[inner sep=0pt]
	\node (a)[point, label=left:$x$] at (1,1) {};
	\node (b)[point, label=left:$y$] at (1,2) {};
	\node (c)[point, label=right:$z$] at (3,2) {};
    \node (d)[point, label=right:$w$] at (3,1) {};

	\draw (a.center) -- (b.center);
	\draw (b.center) -- (c.center);
    \draw (c.center) -- (d.center);
	
	\filldraw [black] (a.center) circle (1pt);
	\filldraw [black] (b.center) circle (1pt);
	\filldraw [black] (c.center) circle (1pt);
 	\filldraw [black] (d.center) circle (1pt);
	\end{tikzpicture}
&
\begin{tikzpicture}
	\tikzstyle{point}=[inner sep=0pt]
	\node (a)[point, label=left:$x$] at (1,1) {};
	\node (b)[point, label=left:$y$] at (1,2) {};
	\node (c)[point, label=right:$z$] at (3,2) {};
    \node (d)[point, label=right:$w$] at (3,1) {};

	\draw (a.center) -- (b.center);
    \draw (c.center) -- (d.center);
	
	\filldraw [black] (a.center) circle (1pt);
	\filldraw [black] (b.center) circle (1pt);
	\filldraw [black] (c.center) circle (1pt);
 	\filldraw [black] (d.center) circle (1pt);
	\end{tikzpicture}
\\
$G_1$ & $G_2$ & $G_3$
\end{tabular}
$$ 
In this example, the free resolution stopped after one step because there were no more relations to consider, but in general, we continue building this sequence using relations between the edges (the second step), then the relations between those relations (the third step), and so on. Hilbert's syzygy theorem guarantees that over a polynomial ring this process stops, so that every free resolution is finite. 

While we know that minimal free resolutions over polynomial rings are finite, constructing them in general is quite challenging. 
When the free resolution is built on relations between monomials (such as edges of a graph), there are concrete methods one could use. One such method is Taylor's resolution~\cite{T}, a free resolution (most often nonminimal) in which, in the case of a graph, at the $i$th step, the monomial indices appearing are the products of vertices of every subset of the edges of size $i$. For example, the free resolution in \eqref{eq:res1} is a Taylor resolution, and for the graphs $G_2$  and $G_3$ above the Taylor resolutions appear on the left and right below, respectively.
$$
0 \longrightarrow 
\begin{array}{c}S(xyzw)\end{array}  
\longrightarrow 
\begin{array}{c}S(xyz)\\ \oplus\\S(yzw) \\\oplus\\S(xyzw)\end{array}  
\longrightarrow 
\begin{array}{c}S(xy)\\ \oplus\\S(yz) \\\oplus\\S(zw)\end{array}
\quad \quad \quad
0 \longrightarrow 
\begin{array}{c}S(xyzw)\end{array}  
\longrightarrow 
\begin{array}{c}S(xy)\\ \oplus\\S(zw)\end{array}
$$
Once we have a resolution, the next question would be to identify which of the multigraded components are redundant, or in other words, to identify the {\it minimal} free resolution. This is where the notion of {\it Scarf multidegrees} comes in: they are the monomials indexing the Taylor resolution that appear exactly once. In the case of $G_1$ and $G_3$, all the monomials are unique, and so the edge ideals of both those graphs have {\it Scarf resolutions}. In the case of $G_2$, however, the monomial $xyzw$ appears twice, once in the second and once in the third step of the Taylor resolution, as 
$$xyzw= xy\cup yz \cup zw=xy \cup zw.$$
In other words, the edge $yz$, which is forming a \say{bridge} between $xy$ and $zw$ in $G_2$, is causing the formation of a non-Scarf multidegree in the Taylor resolution. 

It was shown by  Bayer, Peeva and Sturmfels~\cite{BPS} that all Scarf multidgrees appear in the minimal free resolution of a monomial ideal, though  the minimal resolution may also contain non-Scarf multidegrees. In other words, the Scarf multidegrees may be thought of as a \say{lower bound} for the minimal free resolution of a monomial ideal. Ideals whose minimal resolutions consist of only Scarf multidgrees are said to have \say{Scarf} resolutions. 

The question we ask in this paper is:

\begin{Question} Can one correlate Scarf resolutions with the shape of the graph? What graphs have edge ideals  with Scarf resolutions? What about the powers of those edge ideals? 
\end{Question}

Our main result is the following, and it shows that, in fact, $G_2$ also has a Scarf resolution, but it is not as obvious. 

\begin{theorem*}[{\bf \cref{t:beautiful}: The \say{Beautiful Oberwolfach Theorem}}]  
Let $G$ be a graph with edge ideal $I = I(G)$. 
  \begin{enumerate}
  \item $I$ has a Scarf resolution if and only if $G$ is a gap-free forest.
  \end{enumerate}
   If $G$ is connected and $t>1$, then
  \begin{enumerate}[resume]
  \item $I^t$ has a Scarf resolution if and only if $G$ is an isolated vertex, an edge, or a path of length~$2$.
  \end{enumerate}
\end{theorem*}

Finding combinatorial interpretations of the multidegrees appearing the the minimal free resolution of a monomial ideal is an active area of research. The monomials appearing in a multigraded free resolution can be thought of as labels on faces of simplicial or more generally cell complexes which \say{support} that resolution, for example the Taylor complex or the Scarf complex (see \cref{s:Scarf} for more details).
A central problem in this research area is to construct cellular resolutions: simplicial or cell complexes that \textit{support} the minimal free resolution of a given monomial ideal (cf. \cite{AMFRG2020, BPS, BS, BM2020, BW2002, CK2022, CT2016, EMO, L, OY2015, PV, V}) and their powers~\cite{L2,Lr}. Of particular interest to a graph theorist might be \cite{koszul}, where the authors use a generalization of the box product of graphs to construct cell complexes supporting minimal resolutions of powers of certain monomial ideals.
In general, very little is known about when there are such minimal resolutions even for the special classes
of monomial ideals, those arising as the \textit{edge ideals} of graphs (cf. \cite{BM2020, CK2022}). 

To prove \cref{t:beautiful} for $t=1$, we make use of the characterization of a \textit{gap-free} tree as a graph which contains no induced subgraphs isomorphic to a \textit{triangle}, a \textit{square}, a \textit{pentagon} or a \textit{path} of length 4, and prove that the edge ideals of these particular graphs do not have Scarf resolutions. To establish \cref{t:beautiful} for $t \ge 2$, we show that if $G$ is not an isolated vertex, an edge or a path of length 2, then $G$ contains an induced subgraph that is isomorphic to either a triangle, a square, a path of length three, or a \textit{claw} with three edges, and exhibit that powers of the edge ideals of these special graphs are not supported by their Scarf complexes.

Of our steps in the proof of \cref{t:beautiful} the following are worth highlighting.

\begin{itemize} 
\item{\bf (\cref{t:non-Scarf})} the Scarf
complexes of the edge ideals of triangles, squares, claws and their
powers are constructed explicitly. 

\item \textbf{(\cref{t:tree})} the Scarf complex of the edge ideal of any
forest is completely described.
\end{itemize}
  
This paper is organized as follows. In the next section, we collect 
important facts and terminology about free resolutions and Scarf
complexes. In \cref{s:subgraph}, we look at how the Scarf complex of a graph behaves
when an edge is removed. This allows us to study the Scarf complexes
of subgraphs of a given graph. \cref{s:induced} contains a simple and yet important observation about
Scarf complexes of induced subgraphs, see \cref{l:induced}, then focuses on the
Scarf complex of a graph when a vertex is removed. This allows us to
investigate the Scarf complex of induced subgraphs, and presents a
recursive method to construct the Scarf complex of any graph. In
\cref{s:tree}, we apply results in \cref{s:induced} to
completely describe the Scarf complex of any tree.
\cref{s:ScarfGraphs} is devoted to classifying graphs whose Scarf
complexes support a minimal free resolution of their edge ideals. The
case $t = 1$ of our main result is proved in this
section, see \cref{t:B36}. In \cref{s:Scarf-I2}, we continue our investigation of
graphs for which the powers of the edge ideals have minimal free resolutions
supported by their Scarf complexes, addressing the general case, when
$t \ge 2$, of \cref{t:beautiful}.

\section{Simplicial Resolutions}
\label{s:Scarf}

Throughout this paper, $\sfk$ denotes an arbitrary field and $S=\sfk[x_1,\ldots,x_n]$ is a polynomial ring over $\sfk$. We will identify the variables in $S$ with $n$ distinct points which, by abusing notation, shall also be labeled by $\{x_1, \dots, x_n\}$ and often represent the vertices of a graph.

Let $I \subseteq S$ be a homogeneous ideal. A {\bf free resolution} of $I$ is a (finite) exact sequence of free modules
$$
0 \to S^{c_q} \stackrel{d_q}{\longrightarrow} \cdots \stackrel{d_2}{\longrightarrow} S^{c_1} \stackrel{d_1}{\longrightarrow} S^{c_0}
  $$
where $\im(d_1)=I$. The smallest possible such sequence, that is,
the one with the smallest possible values for the integers $c_i$, is
called a {\bf minimal free resolution } of $I$, and is unique up to
isomorphism of complexes:

\begin{equation}\label{e:mfr}
  0 \to S^{\beta_p} \stackrel{d_p}{\longrightarrow} \cdots
  \stackrel{d_2}{\longrightarrow} S^{\beta_1}
  \stackrel{d_1}{\longrightarrow} S^{\beta_0}.
\end{equation}
Some of the algebraic invariants of $I$ that are visible in
\eqref{e:mfr} are the {\bf Betti numbers} $\beta_i$, and the
length $p$ of the minimal free resolution which is called the {\bf projective dimension} of $I$.

A free resolution of an ideal is essentially built upon the relations
between the generators of the ideal, also known as the {\it syzygies}
of the ideal. In 1966 Taylor~\cite{T} suggested an innovative approach
to constructing a free resolution for a \textit{monomial} ideal $I$ minimally
generated by $q$ monomials, by \say{homogenizing} the chain complex of
a simplex. The process goes as follows:
\begin{itemize}
	\item construct a simplex with $q$ vertices;
	\item label each vertex with one of the monomial generators of $I$;
	\item label each face $\sigma$ with a monomial $\m_\sigma$ which is the least common multiple (lcm)
of the vertex labels of $\sigma$;
	\item use the labels of each face to \say{homogenize} the simplicial chain complex of the $q$-simplex.
\end{itemize}
In her thesis, Taylor proved that this homogenized chain complex is a free
resolution of $I$.

We call the $q$-simplex labeled with the monomial generators of $I$ the
{\bf Taylor complex} of $I$, denoted by $\Taylor(I)$. The
resulting free resolution is called the {\bf Taylor resolution} of $I$.

\begin{example}\label{e:Taylor}
If $I=(xyz, x^2z, xy^2)$, then $\Taylor(I)$ is:
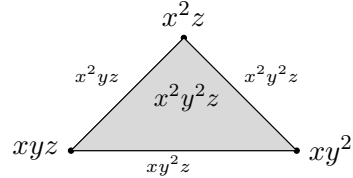
$$
\begin{tikzpicture}
	\tikzstyle{point}=[inner sep=0pt]
	\node (a)[point, label=left:$xyz$] at (1,1) {};
	\node (b)[point, label=above:$x^2z$] at (2.5,2.5) {};
	\node (c)[point, label=right:$xy^2$] at (4,1) {};

	\draw (a.center) -- (b.center);
	\draw (b.center) -- (c.center);
	\draw (a.center) -- (c.center);
	
	\filldraw [black] (a.center) circle (1pt);
	\filldraw [black] (b.center) circle (1pt);
	\filldraw [black] (c.center) circle (1pt);
	
	\draw  [fill=gray!30] (a.center) -- (b.center) -- (c.center) -- cycle;
	
	\pgfputat{\pgfxy(2,0.8)}{\pgfbox[left,center]{\tiny{$ xy^2z$}}}
	\pgfputat{\pgfxy(1.05,2)}{\pgfbox[left,center]{\tiny{$x^2yz$}}}
	\pgfputat{\pgfxy(3.3,2)}{\pgfbox[left,center]{\tiny{$x^2y^2z$}}}
	\pgfputat{\pgfxy(2.1,1.7)}{\pgfbox[left,center]{\small{$x^2y^2z$}}}
\end{tikzpicture}
$$
\end{example}

Since the Taylor resolution is obtained directly by labeling faces and
maps that appear in the a simplicial chain complex of a simplex, the
rank of the free module appearing in each homological degree $i$ will
be the number of $i$-dimensional faces of the simplex. Therefore,
an ideal with $q$ monomial generators in a polynomial ring $S$ has a
Taylor resolution of the following form:
$$ 0 \rightarrow S \rightarrow S^{{q}\choose{q-1}} \rightarrow \cdots
\rightarrow S^{{q}\choose{i}} \rightarrow \cdots \rightarrow S^q.$$

The monomial labels on each face of the Taylor complex allow us to
reinterpret the Taylor resolution as a {\bf multigraded} resolution.
For instance, if the monomial labels of the $i$-dimensional faces of
the Taylor simplex are $\m_1,\ldots,\m_{{q}\choose{i}}$, then in the
$i$-th homological degree of the Taylor resolution, the free
module $S^{{q}\choose{i}}$ can be represented as
$$S(\m_1) \oplus \cdots \oplus S(\m_{{q}\choose{i}}).$$
The monomials appearing in the Taylor complex are least common multiples of the
corresponding generators of $I$, and therefore belong to the {\bf lcm lattice} of
$I$: an atomic lattice denoted by $\LCM(I)$, whose atoms are the
minimal monomial generators of $I$ and the meet of any two elements is
their lcm.

If $I$ is the ideal in \cref{e:Taylor}, then the elements of $\LCM(I)$ are the monomials
$$ xyz, x^2z, xy^2, x^2yz,xy^2z, x^2y^2z$$ which label the faces of  $\Taylor(I)$.

The multigrading of the Taylor resolution is
then inherited by the minimal free resolution, and in particular, the
Betti numbers of $I$ can be written as a sum of multigraded Betti
numbers:
$$\beta_i(I)=\sum_{\m \in \LCM(I)} \beta_{i,\m}(I),$$
where $\beta_{i,\m}(I)$ refers to the number of times the summand $S(\m)$ appears in the
$i$-th homological degree of the multigraded minimal free resolution
of $I$.

It is natural to wonder if a similar construction to the Taylor
complex can be applied to a more general simplicial complex $\Delta$ with $q$
vertices, by homogenizing its simplicial chain complex, in order to
find a free resolution of a monomial ideal with $q$ generators. If
such a resolution exists, then we say that $\Delta$ {\bf supports} a
resolution of $I$. This resolution would be naturally contained in the
Taylor resolution.

Bayer, Peeva and Sturmfels~\cite{BPS, BS} explored this question, and
offered a criterion for a subcomplex $\Delta$ of $\Taylor(I)$ to support a
free resolution of $I$. For such a simplicial complex $\Delta$ and a
monomial $\m$, we use the notation $\Delta_\m$ to denote the induced
subcomplex of $\Delta$ on the vertices who labels divide $\m$. In
other words
\begin{equation}\label{e:induced}
  \Delta_\m =\{ \sigma \in \Delta \st \m_\sigma \mid \m\}
\end{equation}
where $\m_\sigma$ represents monomial label of $\sigma$, or
equivalently the lcm of the monomial labels of the vertices of
$\sigma$.

\begin{theorem}[{\bf Supporting a Free Resolution}~\cite{BPS}]\label{t:BPS}
  A simplicial complex $\Delta$ on $q$ vertices supports a free
  resolution of a monomial ideal $I$ minimally generated by $q$
  monomials in a polynomial ring over a field, if an only if for every
  $\m \in \LCM(I)$ the induced subcomplex $\Delta_\m$, on vertices of
  $\Delta$ whose labels divide $\m$, is empty or acyclic. The resolution
  is minimal if for every pair of faces $\sigma, \tau \in \Delta$ with
  $\sigma \subsetneq \tau$, $\m_\sigma \neq \m_\tau$.
\end{theorem}

The last sentence in \cref{t:BPS} makes it clear why the Taylor
resolution is usually not minimal. In \cref{e:Taylor}, the
monomial label $x^2y^2z$ is shared between a face and a subface, making
the Taylor resolution for the ideal $(xyz, x^2z, xy^2)$ non-minimal.

Naturally one would remove a face and a subface that share a label,
and check, for example using \cref{t:BPS}, if the remaining complex
supports a resolution. An extreme application of this idea is to
remove {\it all} faces which have the same label, regardless of
whether one is embedded in the other or not: for a monomial ideal $I$,
the {\bf Scarf complex} of $I$, denoted by $\Scarf(I)$, is the subcomplex of
$\Taylor(I)$ consisting of all faces whose labels are unique in
$\Taylor(I)$.

\begin{example}
If $I=(xyz, x^2z, xy^2)$ as in \cref{e:Taylor}, the label $x^2y^2z$ is
repeated in $\Taylor(I)$, but all other labels occur on a unique
face. Thus $\Scarf(I)$ is the complex depicted below:
$$
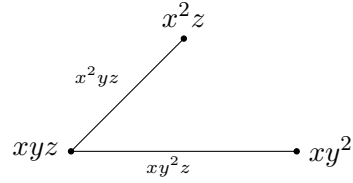
\begin{tikzpicture}
	\tikzstyle{point}=[inner sep=0pt]
	\node (a)[point, label=left:$xyz$] at (1,1) {};
	\node (b)[point, label=above:$x^2z$] at (2.5,2.5) {};
	\node (c)[point, label=right:$xy^2$] at (4,1) {};

	\draw (a.center) -- (b.center);
	\draw (a.center) -- (c.center);
	
	\filldraw [black] (a.center) circle (1pt);
	\filldraw [black] (b.center) circle (1pt);
	\filldraw [black] (c.center) circle (1pt);

	\pgfputat{\pgfxy(2,0.8)}{\pgfbox[left,center]{\tiny{$ xy^2z$}}}
	\pgfputat{\pgfxy(1.05,2)}{\pgfbox[left,center]{\tiny{$x^2yz$}}}
\end{tikzpicture}
$$
\end{example}

The (homogenization of the) Scarf complex of $I$ is contained in every
multigraded free resolution of $I$, but just as the Taylor complex is
often non-minimal, the Scarf complex is often too small to support a
minimal free resolution of $I$.

\begin{definition}[{\bf Scarf Ideals}]
	A monomial ideal $I \subseteq S$ is called a {\bf Scarf ideal} if $\Scarf(I)$ supports a free resolution of $I$. In this case, we also say that $I$ has a {\bf Scarf resolution}.
\end{definition}

Note that if an ideal is Scarf, the resolution supported on $\Scarf(I)$
will necessarily be a minimal free resolution of $I$.

Recall that the {\bf join} $\Delta * \Gamma$ of two simplicial complexes $\Delta$ and $\Gamma$ with disjoint vertex sets is the simplicial complex 
$$\Delta * \Gamma =\{ \sigma \cup \tau \st \sigma \in \Delta, \  \tau \in \Gamma\}.$$

We will make use of the following fact throughout the paper. 

\begin{lemma}\label{l:join} 
If $I$ and $J$ are monomial ideals in disjoint sets of variables, then $$\Scarf(I+J)=\Scarf(I)*\Scarf(J).$$
\end{lemma} 

\begin{proof} Suppose $I$ is generated by monomials in the set of variables $X_1$, and $J$ is generated by monomials in the set of variables $X_2$, where $X_1\cap X_2=\emptyset$. 

If $\sigma \in \Taylor(I)$ and  $\tau \in \Taylor(J)$, then since  $X_1\cap X_2=\emptyset$, $\gcd(\m_\sigma, \m_\tau)=1$, and therefore 
$$\m_{\sigma \cup \tau}=
\lcm(\m_\sigma,\m_\tau)=
\m_\sigma\m_\tau.$$ 

On the other hand, if $\gamma \in \Taylor(I+J)$, then we can write $\gamma=\gamma_1\cup \gamma_2$, where the vertices of $\gamma_1$ are labeled with monomial generators of $I$ (and are hence monomials in $X_1$), and  the vertices of $\gamma_2$ are labeled with monomial generators of $J$ (and are hence monomials in $X_2$). In other words, $\gamma_1 \in \Taylor(I)$ and $\gamma_2\in \Taylor(J)$, and  $\m_\gamma=\m_{\gamma_1}\m_{\gamma_2}$. 

Now, with $\sigma$, $\tau$ and $\gamma$ as above, we have  
\begin{equation}\label{e:disjoint}
\m_\gamma=\m_{\sigma\cup \tau} \iff 
\m_{\gamma_1}\m_{\gamma_2}=\m_\sigma\m_\tau \iff 
\m_{\gamma_1}=\m_\sigma \qand \m_{\gamma_2}=\m_\tau,
\end{equation} 
where the last equalities hold because each pair of multiplied monomials belong to disjoint sets of variables.

To prove the statement of the lemma, observe that if $\sigma \in \Scarf (I)$ and $\tau \in \Scarf (J)$ then by \eqref{e:disjoint} $\sigma \cup \tau$ cannot share a monomial label with any other face of $\Taylor(I+J)$, so $\sigma \cup \tau \in \Scarf(I+J)$.

And conversely, if  $\gamma\in \Scarf(I+J)$, then by the discussion above $\gamma =\gamma_1 \cup \gamma_2$ where 
 $\gamma_1 \in \Taylor(I)$ and $\gamma_2\in \Taylor(J)$. If  $\m_{\gamma_1}=\m_{\sigma}$ for some $\sigma \in \Taylor(I)$, then $\m_{\sigma \cup \gamma_2}=\m_{\gamma_1\cup \gamma_2}$ which  implies that $\sigma\cup \gamma_2=\gamma_1 \cup \gamma_2$ and hence $\sigma=\gamma_1$. Therefore $\gamma_1\in \Scarf(I)$, and a similar argument shows that $\gamma_2 \in \Scarf(J)$, and therefore $\gamma \in \Scarf(I)*\Scarf(J)$. This ends our argument.
\end{proof}

\section{The Scarf complex of edge ideals of graphs}

Our focus in this paper is on the special class of monomial ideals
generated by squarefree monomials of degree~$2$, which are called
\textit{edge ideals}.

\begin{definition}[{\bf Edge Ideal of Graphs}] \label{d:edgeIdeal}
Let $G$ be a simple graph over the vertices $V(G) = \{x_1, \dots, x_n\}$ and with edge set $E(G)$. The {\bf edge ideal} of $G$ is the following square-free monomial ideal
$$I(G)=(x_ix_j \mid \{x_i,x_j\} \in E(G)) \subseteq
\sfk[x_1,\ldots,x_n].$$
\end{definition}
For simplicity of notation, we will use the convention
$$\Taylor(G)=\Taylor(I(G)) \qand
\Scarf(G)=\Scarf(I(G)).$$
Following standard notation, we use $C_n$ to denote a cycle on $n$
vertices and $P_n$ to denote a path of length $n$ on $n+1$ vertices. A
{\bf claw} is defined to be a graph with three edges meeting at a
common vertex. A graph is {\bf gap-free} if whenever $x_1x_2$ and $y_1y_2$ are edges in the same connected component of $G$ then for some choice of $i,j \in \{1,2\}$, $x_iy_j$ is an edge. That is, a connected graph $G$ is gap-free if and only if the {\it induced matching number} of $G$ is 1. We also set $N_G(v) = \{x \in V(G) \mid xv \in E(G)\}$
to denote the {\bf neighborhood} of $v$ in $G$.

\begin{example}\label{e:C4}
  If $G$ is the square $C_4$ with $I(G)=(xy,yz,zw,wx)$, then
  $\Taylor(G)$ is a tetrahedron with many repeated labels, for
  example $$\lcm(xy,yz,zw,wx)=\lcm(xy,yz,zw)=\lcm(xy,zw)=
  \lcm(yz,xw).$$ By removing all faces with repeated labels from the
  tetrahedron, we observe that $\Scarf(G)$ is a also a square, labeled
  as below.

  \begin{center}
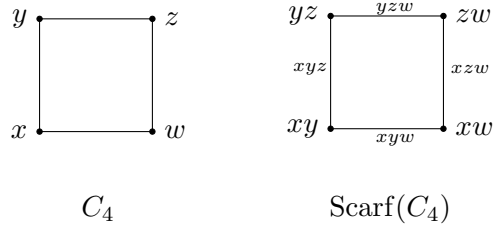
    \begin{tabular}{ccc}
\begin{tikzpicture}
	\tikzstyle{point}=[inner sep=0pt]
	\node (a)[point, label=left:$x$] at (0,0) {};
	\node (b)[point, label=left:$y$] at (0,1.5) {};
	\node (c)[point, label=right:$z$] at (1.5,1.5) {};
	\node (d)[point, label=right:$w$] at (1.5,0) {};

	\draw (a.center) -- (b.center);
	\draw (b.center) -- (c.center);
	\draw (c.center) -- (d.center);
	\draw (d.center) -- (a.center);
	
	\filldraw [black] (a.center) circle (1pt);
	\filldraw [black] (b.center) circle (1pt);
	\filldraw [black] (c.center) circle (1pt);
	\filldraw [black] (d.center) circle (1pt);
\end{tikzpicture}
& \quad\quad &
\begin{tikzpicture}
	\tikzstyle{point}=[inner sep=0pt]
	\node (a)[point, label=left:$xy$] at (0,0) {};
	\node (b)[point, label=left:$yz$] at (0,1.5) {};
	\node (c)[point, label=right:$zw$] at (1.5,1.5) {};
	\node (d)[point, label=right:$xw$] at (1.5,0) {};

	\draw (a.center) -- (b.center);
	\draw (b.center) -- (c.center);
	\draw (c.center) -- (d.center);
	\draw (d.center) -- (a.center);
	
	\filldraw [black] (a.center) circle (1pt);
	\filldraw [black] (b.center) circle (1pt);
	\filldraw [black] (c.center) circle (1pt);
	\filldraw [black] (d.center) circle (1pt);

	\pgfputat{\pgfxy(-.5,0.8)}{\pgfbox[left,center]{\tiny{$xyz$}}}
	\pgfputat{\pgfxy(0.6,1.6)}{\pgfbox[left,center]{\tiny{$yzw$}}}
	\pgfputat{\pgfxy(1.6,0.8)}{\pgfbox[left,center]{\tiny{$xzw$}}}
	\pgfputat{\pgfxy(0.6,-.15)}{\pgfbox[left,center]{\tiny{$xyw$}}}
\end{tikzpicture}
\\
& \quad\quad &\\
$C_4$ & \quad\quad & $\Scarf(C_4)$\\
    \end{tabular}
  \end{center}
\end{example}

\begin{example}\label{ex:disjoint} 
If $I=(xy, yz, uv)$ is the edge ideal of a disconnected graph $G$ on the left, then $\Taylor(I)$ and $\Scarf(I)$ coincide as the simplex on the right. Another way to see this is by \cref{l:join}, which tells us that $\Scarf(I)$ is the join of 
$\Scarf((xy,yz))$ (an edge) and $\Scarf((uv))$ (a point).
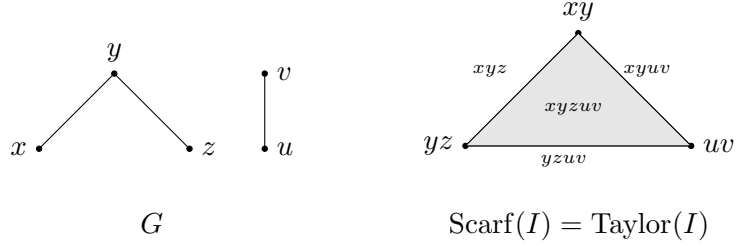
$$
\begin{array}{ccc}
\begin{tikzpicture}
	\tikzstyle{point}=[inner sep=0pt]
	\node (a)[point, label=left:$x$] at (1,1) {};
	\node (b)[point, label=above:$y$] at (2,2) {};
	\node (c)[point, label=right:$z$] at (3,1) {};
    \node (d)[point, label=right:$u$] at (4,1) {};
    \node (e)[point, label=right:$v$] at (4,2) {};

	\draw (a.center) -- (b.center);
    \draw (b.center) -- (c.center);
 	\draw (d.center) -- (e.center);
  
	\filldraw [black] (a.center) circle (1pt);
	\filldraw [black] (b.center) circle (1pt);
	\filldraw [black] (c.center) circle (1pt);
    \filldraw [black] (d.center) circle (1pt);
    \filldraw [black] (e.center) circle (1pt);
	
\end{tikzpicture}& \quad \quad \quad & 
\begin{tikzpicture}
	\tikzstyle{point}=[inner sep=0pt]
	\node (a)[point, label=left:$yz$] at (1,1) {};
	\node (b)[point, label=above:$xy$] at (2.5,2.5) {};
	\node (c)[point, label=right:$uv$] at (4,1) {};

	\draw (a.center) -- (b.center);
	\draw (a.center) -- (c.center);
    \draw (b.center) -- (c.center);
	
	\filldraw [black] (a.center) circle (1pt);
	\filldraw [black] (b.center) circle (1pt);
	\filldraw [black] (c.center) circle (1pt);
	\draw  [fill=gray!20] (a.center) -- (b.center) -- (c.center) -- cycle;
 
	\pgfputat{\pgfxy(2,0.8)}{\pgfbox[left,center]{\tiny{$ yzuv$}}}
	\pgfputat{\pgfxy(1.1,2)}{\pgfbox[left,center]{\tiny{$xyz$}}}
    \pgfputat{\pgfxy(3.1,2)}{\pgfbox[left,center]{\tiny{$xyuv$}}}
    \pgfputat{\pgfxy(2.05,1.5)}{\pgfbox[left,center]{\tiny{$xyzuv$}}}    
\end{tikzpicture}\\ 
&&\\
G && \Scarf(I)=\Taylor(I)
\end{array}
$$
\end{example}

For a face $\sigma \in \Taylor(I)$ and an edge $e$ of $G$ (or minimal
generator $e$ of $I$), we write $e \in \sigma$ if $e$ is among the
vertices appearing in $\sigma$.  We shall often make use of the notion
of distances between edges of graphs and labels in Taylor complexes. To do so, we define a \textit{path} from an edge $e$ to an edge $e'$ as a path connecting a vertex in $e$ to a vertex in $e'$.

\begin{definition}[{\bf Distance}] \label{d:distance}
Let $G$ be a graph, let $e, e' \in E(G)$, and let $\sigma\in \Taylor(G)$. Then the {\bf distance} between $e$ and $e'$, and
  between $\sigma$ and $e'$ are defined, respectively, as
  $$\dist_G(e,e')=\min \{\mbox{number of edges of a path in $G$ connecting $e$ to $e'$}\},$$ and
  $$\dist_G(\sigma,e')=\min\{\dist_G(e,e') \st e \in \sigma\}.$$
\end{definition}
These definitions of distances can be naturally extended to the case where $e'=\{v_1,v_2\}$, for some $v_1, v_2 \in V(G)$, is not necessarily an edge in $G$, by considering $G \cup \{e'\}$ in place of $G$ in \cref{d:distance}.

\begin{example}
    If $G$ is the graph $C_4$ in \cref{e:C4}, $e$ and $e'$ are the edges
  $xy$ and $zw$, respectively, and $\sigma$ is the edge (face) labeled
  $yzw$ in $\Scarf(C4)$, then
  $$\dist_G(e, e') = 1  \qand \dist_G(\sigma, e') = 0.$$
\end{example}

\section{The Scarf complex of a subgraph}\label{s:subgraph}  

We next investigate Scarf complexes of subgraphs; particularly, we
examine how the Scarf complex changes when removing an edge. Let $G$
be a graph, let $vw$ be an edge in $G$, and let $G'$ be the graph
obtained by removing $vw$ from $G$. That is,
\begin{equation}\label{e:notation}
G=G'\cup\{vw\} \qand I(G)=I(G')+(vw).
\end{equation}

Our main result in this section characterizes faces $\sigma \in
\Scarf(G')$ for which $\sigma \cup \{vw\} \in \Scarf(G)$. This is
achieved by combining the following two lemmas.

\begin{lemma}\label{l:B23}
  Let $G$ be a graph and let $vw \in E(G)$. Let $G'$ be the subgraph
  of $G$ obtained by removing the edge $vw$ as in \eqref{e:notation}.
  If $\sigma \in \Scarf(G')$ and $\sigma \cup \{vw\} \in \Scarf(G)$,
  then $\dist_G(e,vw) \neq 1$ for every edge $e \in \sigma$.
\end{lemma}

\begin{proof} Suppose that $e=xy \in \sigma$ and $\dist_G(e,vw) = 1$.
  Then, at least one of $xv,xw,yv,yw$ is an edge of $G$. Without loss
  of generality, assume that $xv\in E(G)$. Set $\tau = \sigma \cup
  \{xv\}$ if $xv \not\in \sigma$, and set $\tau = \sigma \setminus
  \{xv\}$ otherwise. Then, $\sigma \neq \tau$ and 
  $\m_{\sigma\cup \{vw\}}=\m_{\tau \cup \{vw\}}$, and so $\sigma\cup
  \{vw\}\not\in \Scarf(G)$. This is a contradiction. Thus $\dist_G(e,vw) \ne 1$.
\end{proof}

The next lemma provides conditions under which the converse holds. To
state the conditions, we define a {\bf relative neighborhood} for a face $\sigma$ 
of the Taylor complex of $G$ to be 
$$N_{\sigma}(v) = \{ x \in V(G) \mid xv\in E(G) \cap \sigma\}.$$

\begin{lemma}\label{lem.Distance0B25} 
  Let $G$ be a graph, and let $G'$ be its subgraph obtained by removing
  the edge $vw$ as in \eqref{e:notation}. Assume $\sigma \in
  \Scarf(G')$, and $\dist_G(e, vw) \neq 1$ for all $e \in \sigma$.

  \begin{enumerate}
	\item If $\dist_G(\sigma,vw) \geq 2$, then $\sigma \cup \{vw\}
          \in \Scarf(G)$.

	\item If $\dist_G(\sigma, vw) = 0$, then $\sigma \cup \{vw\}
          \in \Scarf(G)$ if and only if 

          \begin{itemize}

\item[(i)] $ N_\sigma(w) = \varnothing$ or $N_\sigma(v) = \varnothing$, and

\item[(ii)] $N_\sigma(w) \cap N_G(v) = \varnothing = N_\sigma(v) \cap
  N_G(w).$
          \end{itemize}
	\end{enumerate}
\end{lemma}

\begin{proof} By definition,
  $\sigma \cup \{vw\} \not\in \Scarf(G)$ if and only if there exists
  $\theta \in \Taylor(G)$ such that $$\theta \neq \sigma \cup \{vw\}
  \qand \lcm(\m_\sigma, vw) = \m_\theta.$$

\noindent (1) Assume that $\lcm(\m_\sigma, vw) = \m_\theta$ for some $\theta \in
\Taylor(G)$. Clearly $v,w \mid \m_{\theta}$.
\begin{itemize}

\item If $vw \not\in \theta$, then since $w \mid \m_\theta$, there
   exists a vertex $a\neq v$ such that $aw \in \theta$, so $a \mid
   \m_{\sigma}$, which means that there is an edge $ab$ of $G$ that is
   in $\sigma$. This implies that we have
   $\dist_G(\sigma,vw) \leq 1$, a contradiction.

\item If $vw \in \theta$,
  then define $\theta' = \theta \setminus \{vw\}$ and note that
  $\lcm(\m_\sigma, vw) = \lcm(\m_{\theta'}, vw).$

  \begin{itemize}
    
  \item If $w \mid \m_{\theta'}$, then since $vw \not\in \theta'$, there
  exists a vertex $a\neq v$ such that $aw \in \theta'$. Then, $a \mid
  \lcm(\m_{\theta'}, vw)$. It follows that $a \mid \m_{\sigma}$, and
  in particular there is an edge $ab$ of $G$ that is in $\sigma$. We
  now have $\dist_G(\sigma,vw) \leq 1$, a contradiction.

  \item If $v \mid \m_{\theta'}$, then we similarly get a contradiction.
  \end{itemize}
  
  Therefore $w,v \nmid \m_{\theta'}$ and $\dist_G(\sigma,vw) \geq 2 $,
  so we have $\m_\sigma = \m_{\theta'}$. Since $\sigma \in
  \Scarf(G')$, $\sigma = \theta'$ and thus $\theta=\sigma \cup
  \{vw\}$, and we are done.

\end{itemize}

\smallskip

\noindent (2) The condition $\dist_G(\sigma, vw) = 0$ is equivalent to either $N_\sigma(w) \neq \varnothing$, $N_\sigma(v) \neq \varnothing$, or both sets are not empty.

\begin{itemize}
\item[($\Longrightarrow$)] Assume $\sigma \cup \{vw\} \in
  \Scarf(G)$. If there are $a \in N_\sigma(v)$ and $b \in
  N_\sigma(w)$, then both $av$ and $bw$ are edges in $\sigma$,
  implying that $vw \mid \m_\sigma$, making $\m_\sigma=\m_{\sigma
    \cup\{vw\}}$, a contradiction to $\sigma \cup \{vw\} \in
  \Scarf(G)$. Thus, we must have either
  \begin{equation}\label{e:dist0}
  N_\sigma(w) \neq \varnothing \text{ and } N_\sigma(v) = \varnothing, \qor N_\sigma(v) \neq \varnothing \text{ and } N_\sigma(w) = \varnothing.
  \end{equation}
  Particularly, (i) holds. In view of \eqref{e:dist0}, we can
  assume, without loss of generality, that
\begin{equation}\label{e:dist-i}
  N_\sigma(w) \neq \varnothing \qand 
  N_\sigma(v) = \varnothing.
\end{equation}
  To show (ii), assume that $a \in
  N_\sigma(w) \cap N_G(v)$ (clearly, $N_\sigma(v) \cap N_G(w) = \varnothing$). Then, both $av$ and $aw$ are edges of $G$, where $aw \in \sigma$ and, by \eqref{e:dist-i}, $av \notin
  \sigma$. This implies that $av\neq wv$ and $\m_{\sigma \cup \{vw\}}=\m_{\sigma \cup
    \{av\}}$, a contradiction to the fact that $\sigma \cup \{vw\} \in \Scarf(G)$.
  Therefore (ii) also holds and we are done.
 
\item[($\Longleftarrow$)] Suppose that conditions (i) and (ii) hold. Without loss of generality, we may assume that $N_\sigma(v) = \varnothing$. This, together with $\dist_G(\sigma, vw) = 0$, forces $N_\sigma(w) \neq \varnothing$. Therefore, we have
  \begin{equation}\label{e:dist0-0}
    N_\sigma(w) \neq \varnothing, \quad N_\sigma(v) = \varnothing, \qand N_\sigma(w)
  \cap N_G(v) = \varnothing.
  \end{equation}
  We will show $\sigma \cup \{vw \} \in \Scarf(G)$.
  Suppose $\lcm(\m_\sigma, vw) = \m_\theta$ for some $\theta \in
   \Taylor(G)$. Then by \eqref{e:dist0-0}, $\m_\theta=v\m_\sigma$, and
   in particular, there is an edge $yv \in \theta$. Since $N_\sigma(w)
   \cap N_G(v) = \varnothing$, $yw \notin \sigma$, so $yz \in \sigma$
   for some $z \notin\{v,w\}$, which implies that $\dist_G(vw,yz)=1$, a
   contradiction.
    \end{itemize}
\end{proof}

Combining the preceding two lemmas allows us to classify all faces of $\Scarf(G)$ that contain a fixed edge $vw$ in terms of Scarf faces of a smaller graph.

\begin{theorem}[{\bf Removing an Edge}]\label{p:B26} 
	Let $G$ be a graph and let $vw$ be an edge in $G$. Set $G' = G \setminus \{vw\}$. Let $\sigma \in \Scarf(G')$. Then,
	$ \sigma \cup \{vw\} \in \Scarf(G)$ if and only if
        $\dist_G(e,vw) \neq 1$ for all $e \in \sigma$ and one of the following condition holds:
        \begin{enumerate} 
        \item $\dist_G(\sigma, vw) \ge 2$, or
        \item $\dist_G(\sigma, vw) = 0$ and
            \begin{enumerate} 
            \item $N_\sigma(w) = \varnothing$ or $N_\sigma(v) = \varnothing$, and
            \item $N_\sigma(w) \cap N_G(v) = N_\sigma(v) \cap N_G(w) =\varnothing.$
            \end{enumerate}
        \end{enumerate}
\end{theorem}

\begin{proof} The assertion follows from \cref{l:B23} and
  \cref{lem.Distance0B25}.
\end{proof}

\cref{p:B26} does not quite give a recursive method to construct the
Scarf complex of an arbitrary graph, as $\Scarf(G')$ is not
necessarily a subcomplex of $\Scarf(G)$. The example below is one of
such a case.

\begin{example}\label{e:B26}
Consider $G'$ and $G$ below.
  $$
   G':	
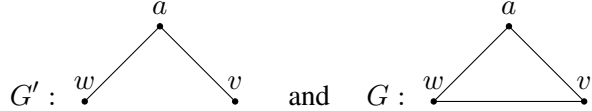
\begin{tikzpicture}
		\tikzstyle{point}=[inner sep=0pt]
		\node (w)[point,label=above:$w$] at (0,0) {};
		\node (a)[point,label=above:$a$] at (1,1) {};
		\node (v)[point,label=above:$v$] at (2,0) {};
		\draw (a.center) -- (w.center);
		\draw (a.center) -- (v.center);
		\filldraw [black] (a.center) circle (1pt);
		\filldraw [black] (v.center) circle (1pt);
		\filldraw [black] (w.center) circle (1pt);
	\end{tikzpicture}
   \qand
   G:	\begin{tikzpicture}
		\tikzstyle{point}=[inner sep=0pt]
		\node (w)[point,label=above:$w$] at (0,0) {};
		\node (a)[point,label=above:$a$] at (1,1) {};
		\node (v)[point,label=above:$v$] at (2,0) {};
		\draw (a.center) -- (w.center);
		\draw (a.center) -- (v.center);
		\draw (v.center) -- (w.center);
		\filldraw [black] (a.center) circle (1pt);
		\filldraw [black] (v.center) circle (1pt);
		\filldraw [black] (w.center) circle (1pt);
	\end{tikzpicture}
$$ Then $\sigma = \{wa, av\} \in \Scarf(G')$, but $\sigma \not\in
   \Scarf(G)$.
\end{example}

The next section focuses on induced subgraphs. Using induced subgraphs will allow us to build the Scarf complex inductively.

\section{The Scarf complex of an induced subgraph} \label{s:induced}

In this section, we study Scarf complexes of \textit{induced}
subgraphs which, unlike results in \cref{s:subgraph}, lead to a
recursive method to construct the Scarf complex of any given graph. Since powers of edge ideals are well-behaved with respect to induced subgraphs, we state the more general case of powers in the first two lemmas, which will prove useful in later sections.

\begin{lemma}\label{l:induced-gens} Let $t$ be a positive integer, $G$ a
  graph and $H$ an induced subgraph of $G$. Then the set of minimal
  monomial generators of $I(H)^t$ is contained in the minimal monomial
  generating set of $I(G)^t$.
\end{lemma}

\begin{proof} Suppose $I(H)$  and $I(G)$ have, respectively, minimal generators
  $$m_1,\ldots,m_q \qand m_1,\ldots,m_q,u_1,\ldots,u_p,$$ where each $m_i$ is an edge
  of $G$ both whose vertices are in $H$, and the monomials $u_i$ correspond to edges
  with at least one vertex outside the vertex set of $H$.
  
  The case $t=1$ is then straightforward. If $t>1$, suppose 
  $\m=m_1^{a_1}\cdots m_q^{a_q}$ is a a minimal
  monomial generator of $I(H)^t$, where $\sum_i a_i=t$ .  If $\m$ is
  not a minimal generator of $I(G)^t$, then $\m' \mid \m$ for some
  $$\m'=m_1^{b_1}\cdots m_q^{b_q} \cdot u_1^{c_1}\cdots u_q^{c_p} \qwhere
  \sum_i b_i + \sum_j c_j=t.$$ If $c_j >0$ for some $j$, then $u_j
  \mid \m$, which is impossible because $u_j$ is divisible by a variable which does not divide
  $\m$. So $c_1=\cdots=c_p=0$, which means that $\m' \in I(H)^t$ and
  $\m' \mid \m$. Given that $\m$ is a minimal generator for $I(H)^t$
  this means that $\m'=\m$, proving our claim.
\end{proof}

The following simple observation allows us to consider Scarf complexes
of induced subgraphs.

\begin{proposition}[{\bf Scarf complex of induced subgraphs}] \label{l:induced}
    Let $G$ be a graph, let $H$ be an induced subgraph of $G$ with no
    isolated vertices, and let $m_H$ be the product of the vertices of
    $H$.  Then for $t\geq 1$
  \begin{enumerate}
  \item  $\LCM((I(H)^t)\subseteq \LCM((I(G)^t)$;
  \item $(m_H)^t \in  \LCM((I(H)^t)$;
  \item $\Scarf(I(H)^t)=\Scarf(I(G)^t)_{(m_H)^t}$ is the induced subcomplex of $\Scarf(I(G)^t)$ on $(m_H)^t$.
  \end{enumerate}
\end{proposition}

\begin{proof}   
  (1) If $M \in \LCM(I(H)^t)$ then there exist minimal generators
  $w_1, \ldots, w_a$ of $I(H)^t$ such that $M=\lcm(w_1, \ldots,
  w_a)$. Since $w_1, \ldots, w_a$ are also minimal generators of
  $I(G)^t$ by \cref{l:induced-gens}, we have $M \in \LCM(I(G)^t)$ as
  well.

  (2)  Suppose $I(H)$  has minimal generators
  $m_1,\ldots,m_q$. Then  
  $m_H=\lcm(m_1,\ldots,m_q)$ and $$(m_H)^t=\lcm \big ( (m_1)^t, \ldots,
  (m_q)^t \big ) \in \LCM(I(H)^t).$$

(3) Consider $\sigma \in \Scarf(I(H)^t)$ and any $\tau \in
  \Taylor(I(G)^t)$. If $\m_{\sigma}=\m_{\tau}$, then for all vertices $v
  \in V(G) \setminus V(H)$, $v \nmid \m_{\tau}$. It follows that $\tau
  \in \Taylor(I(H)^t)$. Since $\sigma \in \Scarf(I(H)^t)$, this
  implies that $\sigma = \tau$. Thus, $\sigma \in \Scarf(I(G)^t)$.  On
  the other hand, $\m_\sigma \mid (m_H)^t$, so $\sigma \in
  \Scarf(I(G)^t)_{(m_H)^t}$.

Now suppose $\sigma \in \Scarf(I(G)^t)_{(m_H)^t}$. Then every
vertex of $\sigma$ has labels which are monomials in $V(H)$, and since
$H$ is an induced subgraph of $G$, these labels belong to
$\LCM(I(H)^t)$, so $\sigma \in \Taylor(I(H)^t)$. Since $\m_\sigma$ is
unique in the $\LCM(I(G)^t) \supseteq \LCM(I(H)^t)$, we have $\sigma \in
\Scarf(I(H)^t)$.
\end{proof}

Let $G$ be a graph and let $v$ be a vertex in $G$,  and let $G\setminus \{v\} $ denote the (induced) subgraph of $G$ obtained by deleting $v$ (and all incident edges) from $G$. That is,
$$G\setminus \{v\} = G\big|_{V(G) \sm \{v\}}.$$

When removing a vertex $v$ from the graph $G$, all edges involving $v$ are removed, thus  building $\Scarf(G)$ from $\Scarf(G')$ will bear a similarity to results in the previous section, but will involve multiple edges containing $v$. Fix  $$\{w_1, \dots, w_t\} \subseteq N_G(v).$$
We begin by identifying faces $\sigma \in \Scarf(G')$ for which $\sigma \cup \{vw_1, \dots, vw_t\} \not\in \Scarf(G)$. To do so, we consider the following cases in three subsequent lemmas:
\begin{enumerate} 
\setlength{\itemindent}{.5in}
	\item[(\cref{lem.dist2B27})] $\dist_G(\sigma, vw_i) \ge 2$ for all $1 \le i \le t$; 
	\item[(\cref{lem.dist1B27})] $\dist_G(e, vw_i) = 1$ for some $e \in \sigma$ and $1 \le i \le t$;
	\item[(\cref{lem.dist0B27})] $\dist_G(\sigma, vw_i) = 0$ for some $1 \le i \le t$ and we are not in \cref{lem.dist1B27}.
\end{enumerate}

\begin{lemma}\label{lem.dist2B27}
Let $G$ be a graph, $v$ a vertex of $G$ and $G'=G \setminus \{v\}$. Suppose $\sigma \in \Scarf(G')$ and $w_1, \dots, w_t \in N_G(v)$. Suppose further that $\dist_G(\sigma, vw_i) \ge 2$ for all $1 \le i \le t$.
\begin{enumerate}
	\item If there exist $1 \le i \not= j \le t$ such that $w_iw_j \in E(G)$ then $$\sigma \cup \{vw_1, \dots, vw_t\} \not\in \Scarf(G' \cup \{vw_1, \dots, vw_t\}).$$
	\item If $w_iw_j \not\in E(G)$ for all $1 \le i\not= j \le t$ then $$\sigma \cup \{vw_1, \dots, vw_t\} \in \Scarf(G' \cup \{vw_1, \dots, vw_t\}).$$
 As a consequence, in this case, we also have $\sigma \cup \{vw_1, \dots, vw_t\} \in \Scarf(G)$.
\end{enumerate}
\end{lemma}

\begin{proof}
(1) Assume $w_iw_j \in E(G)$ for some $i \neq j$. Note that $w_iw_j \not\in \sigma$ since $\dist_G(\sigma, vw_i) \ge 2$. Set $\tau = \sigma \cup \{w_iw_j\}$. 
Clearly, $\lcm(\m_\sigma, vw_1, \dots, vw_t) = \lcm (\m_\tau, vw_1, \dots, vw_t)$. Thus, $\sigma \cup \{vw_1, \dots, vw_t\} \not\in \Scarf(G' \cup \{vw_1, \dots, vw_t\}).$

(2) By \cref{lem.Distance0B25}~(1),
 we have that $\sigma \cup \{vw_1\}
  \in \Scarf(G' \cup \{vw_1\})$. 
Suppose, by induction on $t$, that $\sigma \cup \{vw_1, \dots, vw_{i-1}\} \in \Scarf(G' \cup \{vw_1, \dots, vw_{i-1}\})$ for some $2 \le i \le t$. We shall show that
  \begin{equation} \label{e:4.3ind}
      \sigma \cup \{vw_1, \dots, vw_{i}\} \in \Scarf(G' \cup \{vw_1, \dots, vw_{i}\}).
  \end{equation}
  
For simplicity, set 
$$H = G' \cup \{vw_1, \dots, vw_{i}\}, \qand \tau = \sigma \cup \{vw_1, \dots, vw_{i-1}\}.$$
Observe that                 
\begin{itemize} 
    \item $\dist_H(\tau, vw_{i}) = 0$.
    \item $\dist_H(e, vw_{i}) \ne 1$ for all $e \in \tau$.
    \item $N_\tau(w_i) = \varnothing$, since $\dist_G(\sigma, vw_i) \ge 2$.
    \item $N_\tau(v) = \{w_1, \dots, w_{i-1}\}$, and so $N_H(w_i) \cap N_\tau(v) = \varnothing$ by condition (2).
\end{itemize}
Thus, by applying \cref{lem.Distance0B25}~(2), we arrive at \eqref{e:4.3ind}. For $i = t$, we obtain
$$\sigma \cup \{vw_1, \dots, vw_t\} \in
  \Scarf(G' \cup \{vw_1, \ldots, vw_t\}).$$
  
To prove the last statement, observe that $\sigma \cup \{vw_1, \dots, vw_t\} \not\in \Scarf(G)$ only if there exists $\theta \in \Taylor(G)$ such that $\theta \not= \sigma \cup \{vw_1, \dots, vw_t\}$ and 
$$\m_\theta = \lcm(\m_\sigma, vw_1, \dots, vw_t).$$
Since we have shown that $\sigma \cup \{vw_1, \dots, vw_t\} \in \Scarf(G' \cup \{vw_1, \dots, vw_t\})$, $\theta$ must contain an edge $vw$ for some $w \in N_G(v) \setminus \{w_1, \dots, w_t\}$. This implies that $w \mid \m_\theta$, and so $w \mid \m_\sigma$. Thus, there exists an edge $wx \in \sigma$. In this case, we get $\dist_G(wx, vw_1) \le 1$, a contradiction to the hypothesis. Thus, no such $\theta$ exists, and the statement is proved.
\end{proof}

\begin{lemma}
	\label{lem.dist1B27}
	Let $G$ be a graph, $v$ a vertex of $G$ and $G'=G \setminus \{v\}$.  Let $\sigma \in \Scarf(G')$ and $w_1, \dots, w_t \in N_G(v)$. If there exists an edge $e \in \sigma$ such that $\dist_G(e,vw_i) = 1$ for some $1 \le i \le t$, then $\sigma \cup \{vw_1, \dots, vw_t\} \not\in \Scarf(G).$
\end{lemma}

\begin{proof}
Without loss of generality, suppose that $e = xy$ and $\dist_G(e,vw_1) = 1$. Then, either $xw_1$ or $xv$ is an edge in $G$, and $x,y \not= v$.

Suppose that $xw_1 \in E(G)$. If $xw_1 \in \sigma$ then set $\tau = \sigma \sm \{xw_1\}$. Otherwise, set $\tau = \sigma \cup \{xw_1\}$. In both cases, we end up with $\lcm(\m_\sigma,vw_1, \dots, vw_t) = \lcm(\m_\tau,vw_1, \dots, vw_t)$. Thus, $\sigma \cup \{vw_1, \dots, vw_t\} \not\in \Scarf(G)$.

Suppose now that $xv \in E(G)$. That is, $x \in N_G(v)$. If $x \not\in \{w_1, \dots, w_t\}$ then
$$\lcm(\m_\sigma, vw_1, \dots, vw_t) = \lcm(\m_\sigma, vw_1, \dots, vw_t, vx).$$
Therefore, $\sigma \cup \{vw_1, \dots, vw_t\} \not\in \Scarf(G)$. If $x \in \{w_1, \dots, w_t\}$ then since $\dist_G(e,vw_1) \not= 0$, we have $x \not= w_1$ and $t \ge 2$. In this case,
$$\lcm(\m_\sigma, vw_1, \dots, vw_t) = \lcm(\m_\sigma, \{vw_1, \dots, vw_t\} \sm \{vx\}).$$
Hence, $\sigma \cup \{vw_1, \dots, vw_t\} \not\in \Scarf(G).$
\end{proof}

\begin{lemma}
    \label{lem.dist0B27}
     Let $G$ be a graph, $v$ a vertex of $G$ and $G'=G \setminus \{v\}$. Suppose $\sigma \in \Scarf(G')$ and $w_1, \dots, w_t \in N_G(v)$. Suppose that $\dist_G(\sigma, vw_i) = 0$, for some $1 \le i \le t$, and $\dist_G(e, vw_j) \not= 1$ for any $e \in \sigma$ and $1 \le j \le t$. Then 
    $$\sigma \cup \{vw_1, \dots, vw_t\} \in \Scarf(G' \cup \{vw_1, \dots, vw_t\}) \Longleftrightarrow t = 1.$$
\end{lemma}

\begin{proof}  Without loss of generality, assume that $\dist_G(\sigma, vw_1) = 0$.
    If $t \ge 2$, then  there exists an edge $aw_1 \in \sigma$, and so 
    $$\lcm(\m_\sigma, vw_1, \dots, vw_t) = \lcm(\m_\sigma, vw_2, \dots, vw_t).$$
    Thus,
    $\sigma \cup \{vw_1, \dots, vw_t\} \not\in \Scarf(G' \cup \{vw_1, \dots, vw_t\}).$

    Conversely, suppose that $t = 1$. To prove that $\sigma \cup \{vw_1\} \in \Scarf(G' \cup \{vw_1\})$, it suffices to show that if $\theta \in \Taylor(G' \cup \{vw_1\})$ is such that 
    $\lcm(\m_\sigma, vw_1) = \m_\theta$
    then $\theta = \sigma \cup \{vw_1\}$.  

   Since $v \mid \m_\theta$, we must have $vw_1 \in \theta$. Set $\theta' = \theta \setminus \{vw_1\}$. 
   
    Suppose $w_1 \nmid \m_{\theta'}$. Then  since $\dist_G(\sigma, vw_1) = 0$, there must exists an edge $aw_1 \in \sigma$. This implies that $a \mid \m_{\theta'}$, and so there is an edge $ab \in \theta'$, where $b \not= w_1$. Since $b \mid \m_{\theta'}$, $b \mid \m_\sigma$.  There are two cases to consider.
    
      \begin{itemize}
       \item  If $bw_1 \in E(G)$, set $\tau = \sigma \setminus \{bw_1\}$ if $bw_1 \in \sigma$, or $\tau = \sigma \cup \{bw_1\}$ otherwise. Then $\m_\sigma = \m_\tau$, a contradiction to the fact that $\sigma \in \Scarf(G')$. 
      
      \item If $bw_1 \not\in E(G)$, it follows that $\dist_G(ab, vw_1) = 1$. Particularly, $ab \not\in \sigma$.
       Then $\m_\sigma = \m_{\sigma \cup \{ab\}}$, again a contradiction to the fact that $\sigma \in \Scarf(G')$.

    \end{itemize} 
    
    So we must have $w_1 \mid \m_{\theta'}$, then
    $\m_\sigma = \m_{\theta'}$. This, together with the fact that $\sigma \in \Scarf(G')$ forces $\sigma = \theta'$, whence $\sigma \cup \{vw_1\} = \theta$, and we are done.  
\end{proof}

\begin{lemma}
    \label{rem.dist0B27_extra}
    Assume the same hypothesis as in \cref{lem.dist0B27}. The following are equivalent:
   \begin{enumerate} 
    \item $\sigma \cup \{vw_1, \dots, vw_t\} \in \Scarf(G)$.
    \item $t = 1$ and $N_\sigma(w_1) \cap N_G(v) = \varnothing.$
    \end{enumerate}
\end{lemma}

\begin{proof}
    We first prove (1) $\Longrightarrow$ (2). Assume that $\sigma \cup \{vw_1, \dots, vw_t\} \in \Scarf(G)$. By essentially the same proof as in \cref{lem.dist0B27}, we can show that $t = 1$. If $N_G(v) = \{w_1\}$ then (2) is established. 
    
    Suppose that $|N_G(v)| \ge 2$, and consider any $w \in N_G(v) \setminus \{w_1\}$. If $w_1w \in \sigma$, then
    $$\lcm(\m_\sigma, vw_1) = \lcm(\m_\sigma, vw_1, vw),$$
    a contradiction to the assumption that $\sigma \cup \{vw_1\} \in \Scarf(G)$. Hence, $N_\sigma(w_1) \cap N_G(v) = \varnothing.$
    
    We proceed to prove (2) $\Longrightarrow$ (1). Assume that $t=1$ and $N_\sigma(w_1) \cap N_G(v) = \varnothing$. Observe that $\sigma \cup \{vw_1\} \not\in \Scarf(G)$ only if there exists $\theta \in \Taylor(G)$ such that $\theta \not= \sigma \cup \{vw_1\}$ and $\m_\theta = \lcm(\m_\sigma, vw_1)$. Since, by \cref{lem.dist0B27}, $\sigma \cup \{vw_1\} \in \Scarf(G' \cup \{vw_1\})$, $\theta$ must contain an edge $vw$, for some $w \in N_G(v) \setminus \{w_1\}$. 

    Since $w \mid \m_\theta$, we have $w \mid \m_\sigma$. Thus, there is an edge $wx \in \sigma$. Since $N_\sigma(w_1) \cap N_G(v) = \varnothing$, we must have $x \not= w_1$. This implies that $\dist_G(wx, vw_1) = 1$, a contradiction to the hypothesis. Hence, (1) is established, and the lemma is proved.
\end{proof}

We now arrive at the main result of this section. This result provides a recursive algorithm for constructing the Scarf complex of the edge ideal of any graph.

\begin{theorem}[{\bf Removing a Vertex}]\label{t:B29}
	Let $G$ be a graph and let $v \in V(G)$. Set $G' = G \setminus \{v\}$.
	Let $\tau \in \Taylor(G)$. Then, $\tau \in \Scarf(G)$ if and only if the following conditions are satisfied:
	\begin{enumerate}
		\item $\sigma = \tau\big|_{G'} \in \Scarf(G')$, and
		\item $\tau = \sigma \cup \{vw_1, \dots, vw_t\}$, where $w_1, \dots, w_t \in N_G(v)$, and either
		\begin{enumerate}
            \item $\dist_G(\sigma, vw_i) \ge 2$, for all $1 \le i \le t$, and $w_iw_j \not\in E(G)$ for all $1 \leq i <j \leq t$; or      
            \item $\dist_G(\sigma, vw_1) = 0$, $t = 1$, $N_\sigma(w_1) \cap N_G(v) = \varnothing$, and $\dist_G(e, vw_1) \not= 1$ for all $e \in \sigma$.
		\end{enumerate}
	\end{enumerate}
\end{theorem}

\begin{proof} By \cref{l:induced}, if $\tau \in \Scarf(G)$ then
  $\sigma = \tau\big|_{G'} \in \Scarf(G')$. For any $\tau \in
  \Taylor(G)$, we can always write $\tau = \sigma \cup \{vw_1, \dots,
  vw_t\}$, where $\sigma = \tau\big|_{G'}$ and $w_1, \dots, w_t \in
  N_G(v)$. The assertion now follows from  \cref{lem.dist2B27}, \cref{lem.dist1B27},
 \cref{lem.dist0B27}, and \cref{rem.dist0B27_extra}.
\end{proof}

\section{The Scarf complex of a forest} \label{s:tree}

In this section, we apply results in \cref{s:induced} when the deleted vertex is a leaf. We then focus on trees and more generally forests. Using the fact that deleting a leaf of a tree (or forest) results in a smaller tree (or forest), the results give a recursive algorithm for the computation of the Scarf complex of any tree or forest. We also give a direct method of computing the Scarf complex of a forest. 

Throughout the section, unless otherwise stated, we shall assume $G$ is a graph with a
leaf vertex $v \in V(G)$ with the associated edge $vw \in E(G)$. Set
$G' = G \setminus \{v\}$. That is,
\begin{equation}\label{e:leaf} 
G=G'\cup\{vw\} \qand I(G)=I(G')+(vw).
\end{equation}

\begin{theorem} \label{t:B14}
  Let $G$ be a graph with leaf $v$. With notation as in \eqref{e:leaf}, let $\sigma \in \Scarf(G')$. The following are equivalent:
 \begin{itemize}
 	\item[$(a)$] $\sigma \cup \{vw\} \in \Scarf(G)$.
 	\item[$(b)$]  $\forall e \in \sigma, \dist_G(e,vw) \neq 1$.
 \end{itemize}
\end{theorem}

\begin{proof}
    The implication (a) $\Longrightarrow$ (b) follows from \cref{lem.dist1B27}. On the other hand, if $\dist_G(e,vw) \not=1$ for all $e \in \sigma$, then either $\dist_G(\sigma, vw) \ge 2$ or $\dist_G(\sigma,vw) = 0$. If $\dist_G(\sigma,vw) = 0$ then, since $v$ is a leaf in $G$, $N_G(v) = \{w\}$. Thus, the implication (b) $\Longrightarrow$ (a) follows from \cref{t:B29}.
\end{proof}

\begin{example}\label{e:B18}
	Let $G$ be the graph depicted below corresponding to the edge ideal $I=(ab,bc,bd,de)$.
	$$
	
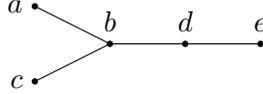
\begin{tikzpicture}
		\tikzstyle{point}=[inner sep=0pt]
		\node (a)[point,label=left:$a$] at (1,2) {};
		\node (b)[point,label=above:$b$] at (2,1.5) {};
		\node (c)[point,label=left:$c$] at (1,1) {};
		\node (d)[point,label=above:$d$] at (3,1.5) {};
		\node (e)[point,label=above:$e$] at (4,1.5) {};
		
		\draw (a.center) -- (b.center);
		\draw (c.center) -- (b.center);
		\draw (b.center) -- (d.center);
		\draw (d.center) -- (e.center);
	
		\filldraw [black] (a.center) circle (1pt);
		\filldraw [black] (b.center) circle (1pt);
		\filldraw [black] (c.center) circle (1pt);
		\filldraw [black] (d.center) circle (1pt);
		\filldraw [black] (e.center) circle (1pt);
		
	\end{tikzpicture}
	$$
	Then $\Taylor(G)$ is a simplex on $4$ vertices. It is easy to verify that, since $\lcm(ab,de)=\lcm(ab,bd,de)$ and $\lcm(bc,de)=\lcm(bc,bd,de)$ that the Scarf complex of $G$ is:
	$$
	
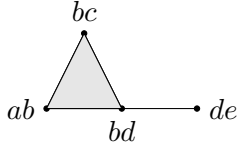
\begin{tikzpicture}
		\tikzstyle{point}=[inner sep=0pt]
		\node (ab)[point,label=left:$ab$] at (1,1) {};
		\node (bc)[point,label=above:$bc$] at (1.5,2) {};
		\node (bd)[point,label=below:$bd$] at (2,1) {};
		\node (de)[point,label=right:$de$] at (3,1) {};
	
		\draw (ab.center) -- (bc.center);
		\draw (ab.center) -- (bd.center);
		\draw (bc.center) -- (bd.center);
		\draw (bd.center) -- (de.center);
	
		\filldraw [black] (ab.center) circle (1pt);
		\filldraw [black] (bc.center) circle (1pt);
		\filldraw [black] (bd.center) circle (1pt);
		\filldraw [black] (de.center) circle (1pt);
		
		\draw  [fill=gray!20] (ab.center) -- (bc.center) -- (bd.center) -- cycle;
		
	\end{tikzpicture}
	$$
	If $de$ plays the role of $wv$ in the theorem, then $G'$ corresponds to a claw, whose Scarf complex is a filled triangle. The only face of this complex satisfying the conditions of \cref{t:B14} is $\{bd\}$ since $\dist_G(ab,de)=\dist_G(cb,de)=1$.
\end{example}

  With notation as in \eqref{e:leaf}, let $G''$ be the induced subgraph of $G$ on vertices of distance $\geq 2$ from $w$, and let $\sigma \in \Scarf(G'')$. Then every edge of $G''$ has distance at least $2$ from $vw$, and so by  \cref{lem.Distance0B25} and \cref{l:induced}, we have 
$\sigma \cup \{vw\} \in \Scarf(G)$.

\begin{example}\label{e:B11}
Let $G$ be the path on 6 vertices, labeled $a$ through $f$
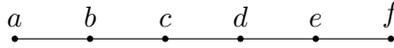
$$
	\begin{tikzpicture}
		\tikzstyle{point}=[inner sep=0pt]
		\node (a)[point,label=above:$a$] at (1,0) {};
		\node (b)[point,label=above:$b$] at (2,0) {};
		\node (c)[point,label=above:$c$] at (3,0) {};
		\node (d)[point,label=above:$d$] at (4,0) {};
		\node (e)[point,label=above:$e$] at (5,0) {};
		\node (f)[point,label=above:$f$] at (6,0) {};
		\draw (a.center) -- (b.center);
		\draw (b.center) -- (c.center);
		\draw (c.center) -- (d.center);
		\draw (d.center) -- (d.center);
		\draw (d.center) -- (e.center);
		\draw (e.center) -- (f.center);
		\filldraw [black] (a.center) circle (1pt);
		\filldraw [black] (b.center) circle (1pt);
		\filldraw [black] (c.center) circle (1pt);
		\filldraw [black] (d.center) circle (1pt);
		\filldraw [black] (e.center) circle (1pt);
		\filldraw [black] (f.center) circle (1pt);
	\end{tikzpicture}
$$
and $I=I(G)$. Then $\Taylor(G)$ is a $4$-dimensional simplex on $5$ vertices, labeled by the generators of $I=(ab,bc,cd,de,ef)$. Using the notation above with $v=f$, we have $w=e$, $G'$ is the path from $a$ to $e$ and $G''$ is the path from $a$ to $c$.

By computing the labels of all possible faces of $\Taylor(G)$ and identifying faces whose labels are unique, it is straightforward to verify that $\Scarf(G)$ is the complex with vertices $ab,bc,cd,de,ef$, edges labeled $abc, bcd,cde,def,abde,abef,bcef$, and triangles $abcef, abdef$ depicted below.
$$
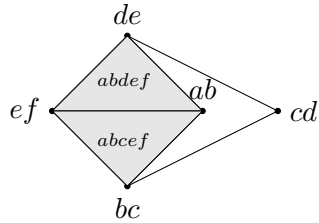
\begin{tikzpicture}
	\tikzstyle{point}=[inner sep=0pt]
	\node (ab)[point,label=above:$ab$] at (2,1) {};
	\node (bc)[point,label=below:$bc$] at (1,0) {};
	\node (cd)[point,label=right:$cd$] at (3,1) {};
	\node (de)[point,label=above:$de$] at (1,2) {};
	\node (ef)[point,label=left:$ef$] at (0,1) {};
	
	\draw (ab.center) -- (de.center);
	\draw (ab.center) -- (bc.center);
	\draw (ab.center) -- (ef.center);
	\draw (de.center) -- (ef.center);
	\draw (bc.center) -- (ef.center);
	\draw (cd.center) -- (de.center);
	\draw (cd.center) -- (bc.center);
	\filldraw [black] (ab.center) circle (1pt);
	\filldraw [black] (bc.center) circle (1pt);
	\filldraw [black] (cd.center) circle (1pt);
	\filldraw [black] (de.center) circle (1pt);
	\filldraw [black] (ef.center) circle (1pt);
	
	\draw  [fill=gray!20] (ab.center) -- (de.center) -- (ef.center) -- cycle;
	\draw  [fill=gray!20] (ab.center) -- (bc.center) -- (ef.center) -- cycle;

	\pgfputat{\pgfxy(.6,1.4)}{\pgfbox[left,center]{\tiny{$ abdef$}}}
	\pgfputat{\pgfxy(.6,.6)}{\pgfbox[left,center]{\tiny{$abcef$}}}
\end{tikzpicture}
$$
Now, $\Scarf(G'')=\Taylor(G'')$ consists of two vertices and the edge $\sigma = \{ab,bc\}$. Note that $\sigma \cup \{ef\} \in \Scarf(G)$, corresponding to the lower of the two shaded triangles in the picture above. The other triangle in $\Scarf(G)$ is explained by the next statement.
\end{example}
We now restrict our attention to the case of forests. While the earlier theorems give a recursive algorithm for computing the Scarf complex, the next theorem provides a direct way to compute the Scarf complex when the graph is a forest.

\begin{theorem}[{\bf Scarf Complexes of Forests}]\label{t:tree}
  Let $G$ be a forest with $q$ edges. Let $K_q$ be the complete graph on
  $q$ vertices each labeled with an edge of $G$. Let $K'$ be the
  subgraph of $K_q$ obtained by removing edges $\{e,e'\}$ from $K_q$,
  where $e$ and $e'$ are edges of $G$ with $\dist_G(e,e')=1$.  Then
  $\Scarf(G)$ is the clique complex of $K'$.
\end{theorem}

\begin{proof} We fix an ordering $e_1, \ldots,
  e_q$ on the edges of $G$ so that each $e_i$ is a leaf edge of
        the forest whose edges are $\{ e_1, \ldots,
        e_{i}\}$.

        We use induction on $q$ to show that
        \begin{equation}\label{e:B20}
          \sigma \in \Scarf(G)
          \iff
  \sigma=\{e_{i_1},\ldots,e_{i_t}\}
        \end{equation}
        for some $i_1 <i_2< \cdots < i_t$ where
  $$\dist_G(e_{i_j},e_{i_k}) \neq 1 \qfor 1 \leq k <j \leq t.$$
The base case  $q = 1$ is trivial, since $\Scarf(G)$ will
        consist of a single point in this case.
	
	Suppose that $q \geq 2$. Assume that
        $\{e_{i_1},\ldots,e_{i_t}\}$ satisfies the condition that
	$$i_1 <i_2< \cdots < i_t \qand  \dist_G(e_{i_j},e_{i_k})
        \neq 1 \qfor 1 \leq k <j \leq t.$$ Note that if $H$ is the induced subgraph of $G$ with vertices the endpoints of the edges $e_{i_j}$, then the edges of $H$ are precisely $\{e_{i_1},\ldots,e_{i_t}\}$ since the distance between any two of these edges is not one.
        If $i_t < q$, then by
        induction, noting that the set $\{e_{i_1},\ldots,e_{i_t}\}$
        satisfies the necessary conditions relative to the induced
        subtree $H$,
        we have that $\{e_{i_1},\ldots,e_{i_t}\}$ is a face of
        $\Scarf(H) \subseteq \Scarf(G)$ by \cref{l:induced}. If $i_t =
        q$, then $e_{i_t}$ is a leaf edge of $G$, and so set $H = G
        \setminus {e_q}$ and note that
        $\{e_{i_1},\ldots,e_{i_{t-1}}\}\in \Scarf(H)$ by
        induction. Thus, by \cref{t:B14},
        $\{e_{i_1},\ldots,e_{i_t}\}\in \Scarf(G)$.
	
	Conversely, if $\{e_{i_1},\ldots,e_{i_t}\} \in \Scarf(G)$,
        then $\dist_G(e_{i_j}, e_{i_k}) \neq 1$ for all $j \neq
        k$. Indeed, to see this, assume that $\dist_G(e_{i_j}, e_{i_k})
        = 1$ for some $j\neq k$. Set $e_{i_j} = \{x,y\}$ and $e_{i_k}
        = \{a,b\}$. Without loss of generality, we may assume also
        that $\{x,a\}$ is an edge in $G$. Set $\sigma =
        \{e_{i_1},\ldots,e_{i_t}\}$ and $\tau = \sigma \cup \{x,a\}$
        if $\{x,a\} \not\in \{e_{i_1},\ldots,e_{i_t}\}$ and $\tau =
        \sigma \setminus \{x,a\}$ else. Note that
        $\m_\sigma = \m_\tau$, so $\sigma$ is not a face of
        $\Scarf(G)$.

       Now observe that the labeled graph $K_q$ is the $1$-skeleton of
       $\Taylor(G)$. By \eqref{e:B20} the faces of $\Scarf(G)$ are
       exactly the cliques of $K'$.
\end{proof}

It can be seen that, in \cref{t:tree}, maximal cliques of $K'$ correspond to facets of $\Scarf(G)$. The following example gives an instance where the Scarf complex of a tree does not support a resolution of its edge ideal.

\begin{example} \label{e:B21}
	Let $G$ be the tree with $5$ edges depicted below.
	$$
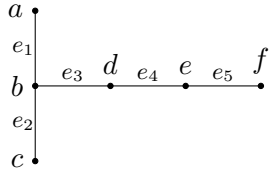
\begin{tikzpicture}
	\tikzstyle{point}=[inner sep=0pt]
	\node (a)[point,label=left:$a$] at (0,2) {};
	\node (b)[point,label=left:$b$] at (0,1) {};
	\node (c)[point,label=left:$c$] at (0,0) {};
	\node (d)[point,label=above:$d$] at (1,1) {};
	\node (e)[point,label=above:$e$] at (2,1) {};
	\node (f)[point,label=above:$f$] at (3,1) {};
	
	\draw (a.center) -- (b.center);
	\draw (b.center) -- (c.center);
	\draw (b.center) -- (d.center);
	\draw (d.center) -- (e.center);
	\draw (e.center) -- (f.center);
		
	\filldraw [black] (a.center) circle (1pt);
	\filldraw [black] (b.center) circle (1pt);
	\filldraw [black] (c.center) circle (1pt);
	\filldraw [black] (d.center) circle (1pt);
	\filldraw [black] (e.center) circle (1pt);
	\filldraw [black] (f.center) circle (1pt);
	
	\node [point,label=left:{\footnotesize $e_1$}] at (.15,1.5) {};
	\node [point,label=left:{\footnotesize $e_2$}] at (.15,.5) {};
	\node [point,label=above:{\footnotesize $e_3$}] at (.5,.9) {};
	\node [point,label=above:{\footnotesize $e_4$}] at (1.5,.9) {};
	\node [point,label=above:{\footnotesize $e_5$}] at (2.5,.9) {};
	
\end{tikzpicture}
$$	
Then $\Scarf(G)$ is the following complex:
	$$
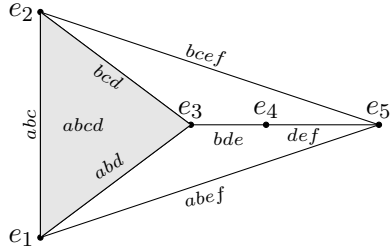
\begin{tikzpicture}
	\tikzstyle{point}=[inner sep=0pt]
\node (a)[point] at (0,-.5) {};
\node (b)[point] at (0,2.5) {};
\node (c)[point] at (2,1) {};
\node (d)[point] at (3,1) {};
\node (e)[point] at (4.5,1) {};

\node[xshift=-.25cm] at (a) {$e_1$};
\node[xshift=-.25cm] at (b) {$e_2$};
\node[yshift=.2cm] at (c) {$e_3$};
\node[yshift=.2cm] at (d) {$e_4$};
\node[yshift=.2cm] at (e) {$e_5$};

\draw (a.center) -- (b.center);
\draw (b.center) -- (c.center);
\draw (a.center) -- (c.center);
\draw (c.center) -- (d.center);
\draw (d.center) -- (e.center);
\draw (a.center) -- (e.center);
\draw (b.center) -- (e.center);

\filldraw [black] (a.center) circle (1pt);
\filldraw [black] (b.center) circle (1pt);
\filldraw [black] (c.center) circle (1pt);
\filldraw [black] (d.center) circle (1pt);
\filldraw [black] (e.center) circle (1pt);

	\draw  [fill=gray!20] (a.center) -- (b.center) -- (c.center) -- cycle;
	
	\node [point,label=left:\rotatebox{90}{\tiny $abc$}] at (.1,1) {};
	\node [point,label=left:{\tiny $abcd$}] at (1,1) {};
	\node [point,label=above:\rotatebox{32}{\tiny $abd$}] at (.9,0.1) {};
	\node [point,label=above:\rotatebox{-32}{\tiny $bcd$}] at (.9,1.3) {};
	\node [point,label=above:{\tiny $bde$}] at (2.5,.57) {};
	\node [point,label=above:{\tiny $def$}] at (3.5,.57) {};
	\node [point,label=above:\rotatebox{-12}{\tiny $bcef$}] at (2.2,1.6) {};
	\node [point,label=below:\rotatebox{12}{\tiny $abef$}] at (2.2,.35) {};
\end{tikzpicture}
$$	
where the triangle with vertices $e_1,e_2,e_5$ is also included.

Observe that $\Scarf(G)$ is not acyclic, since for example $e_1, e_3,e_4,e_5$ forms a non-trivial cycle. Thus by \cite{BPS}, $\Scarf(G)$ does not support a resolution of $I(G)$.
\end{example}

\begin{example}
Let $G$ be the tree depicted below, whose edges are $ab,bc,bd,de,df,dg$.
	$$
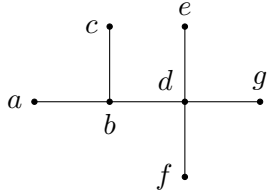
\begin{tikzpicture}
	\tikzstyle{point}=[inner sep=0pt]
	\node (a)[point,label=left:$a$] at (0,1) {};
	\node (b)[point,label=below:$b$] at (1,1) {};
	\node (c)[point,label=left:$c$] at (1,2) {};
	\node (d)[point,label=above left:$d$] at (2,1) {};
	\node (e)[point,label=above:$e$] at (2,2) {};
	\node (f)[point,label=left:$f$] at (2,0) {};
	\node (g)[point,label=above:$g$] at (3,1) {};
	
	\draw (a.center) -- (b.center);
	\draw (b.center) -- (c.center);
	\draw (b.center) -- (d.center);
	\draw (d.center) -- (e.center);
	\draw (d.center) -- (f.center);
	\draw (d.center) -- (g.center);
			
	\filldraw [black] (a.center) circle (1pt);
	\filldraw [black] (b.center) circle (1pt);
	\filldraw [black] (c.center) circle (1pt);
	\filldraw [black] (d.center) circle (1pt);
	\filldraw [black] (e.center) circle (1pt);
	\filldraw [black] (f.center) circle (1pt);
	\filldraw [black] (g.center) circle (1pt);
	
\end{tikzpicture}
$$
Then, by \cref{t:tree}, starting with the complete graph on vertices $\{ ab, bc, bd,de,df,dg\}$ that form the $1$-skeleton of the Taylor simplex, edges $\{ab, de\}, \{ab, df\}
, \{ab,dg\},\{bc,de\}, \{bc,df\},\{bc,dg\}$ are removed, leaving the graph:
	$$
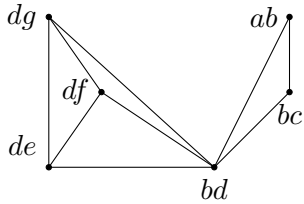
\begin{tikzpicture}
	\tikzstyle{point}=[inner sep=0pt]
	\node (ab)[point,label=left:$ab$] at (3,2) {};
	\node (bc)[point,label=below:$bc$] at (3,1) {};
	\node (bd)[point,label=below:$bd$] at (2,0) {};
	\node (de)[point,label=above left:$de$] at (-.2,0) {};
	\node (df)[point,label=left:$df$] at (.5,1) {};
	\node (dg)[point,label=left:$dg$] at (-.2,2) {};

	\draw (ab.center) -- (bc.center);
	\draw (bc.center) -- (bd.center);
	\draw (ab.center) -- (bd.center);
	\draw (dg.center) -- (bd.center);
	\draw (dg.center) -- (df.center);
	\draw (dg.center) -- (de.center);
	\draw (df.center) -- (de.center);
	\draw (df.center) -- (bd.center);	
	\draw (de.center) -- (bd.center);
		
	\filldraw [black] (ab.center) circle (1pt);
	\filldraw [black] (bc.center) circle (1pt);
	\filldraw [black] (bd.center) circle (1pt);
	\filldraw [black] (de.center) circle (1pt);
	\filldraw [black] (df.center) circle (1pt);
	\filldraw [black] (dg.center) circle (1pt);
	
\end{tikzpicture}
$$
The maximal cliques of this graph are the tetrahedron $\{bd,de,df,dg\}$ and the triangle $\{ab,bc,bd\}$, which are the facets of the Scarf complex of the tree. Note that this complex is acyclic and supports a minimal resolution of $I(G)$. 
\end{example}

\section{Graphs which are Scarf} \label{s:ScarfGraphs}

This section is devoted to the first part of our \say{Beautiful Oberwolfach Theorem}, when the power $t$ is 1. Particularly, we shall characterize all graphs whose Scarf complexes support a resolution of their edge ideals.

We start with a lemma identifying \say{forbidden} subgraph structures that prevent a graph from being Scarf.

\begin{lemma}\label{lem.lem1B36}
Let $G$ be a graph. Suppose that $G$ contains an induced subgraph $H$, which
  belongs to $\{C_3, C_4, C_5, P_4\}$. Then, $G$ is not Scarf.
\end{lemma}

\begin{proof}
	Let $m_H$ be the product of the vertices of $H$.  If $H$ is
        $C_3$, $C_4$, $C_5$ or $P_4$, then a direct computation as in
        \cref{e:C4} shows that $\Scarf(H)$ is $3$ isolated vertices,
        $C_4$, $C_5$, or $C_4$, respectively. Each of these latter
        complexes has nontrivial homology, and so $H$ is not
        Scarf. \cref{l:induced} now implies that $\Scarf(G)_{m_H}=\Scarf(H)$ is not acyclic. Hence, $G$ is not Scarf by \cref{t:BPS}.
\end{proof}

The next lemma gives a better understanding of graphs without the forbidden subgraphs listed in \cref{lem.lem1B36}.

\begin{lemma}\label{lem.lem2B36}
Let $G$ be a graph. The following are equivalent:
  \begin{enumerate}
  \item $G$ is a gap-free forest;
  \item $G$ is does not
        contain an induced subgraph isomorphic to one of $C_3$,
        $C_4$, $C_5$, or $P_4$.
  \end{enumerate}
\end{lemma}

\begin{proof} Assume $G$ is a gap-free forest. Since $G$ is a forest it
  cannot contain an induced cycle. Since $G$ is gap-free, it does
  contain any pair of edges of distance $2$, and thus cannot contain
  an induced $P_4$.

	For the converse, assume $G$ is not a gap-free forest. Then
        either $G$ is not a forest or $G$ is not gap-free. Assume first
        that $G$ is not a forest. Then $G$ contains an induced cycle
        $C$. If the size of $C$ is at most $5$ then $C$ belongs to
        $\{C_3, C_4, C_5\}$. On the other hand, if the size of $C$ is
        greater than or equal to $6$ then $C$ contains an induced
        subgraph which is isomorphic to $P_4$, and thus so does $G$.

        Now assume $G$ is not gap-free. Then there are two edges $e_1$
        and $e_2$ in the same connected component of $G$ whose induced subgraph does not contain any
        additional edges. It follows that the distance between
        $e_1$ and $e_2$ is at least $2$. So $G$ contains a subgraph
        $H$ isomorphic to $P_4$. If $H$ is not an induced subgraph,
        then the induced subgraph on the vertices of $H$ contains an
        induced cycle of length at most $5$,  contradicting~(2).
        \end{proof}

We are now ready to state the first part of our main result, the \say{Beautiful Oberwolfach Theorem}.

\begin{theorem}[{\bf Scarf Graphs are Gap-Free Forests}]\label{t:B36}
Let $G$ be a graph. The edge ideal of $G$ has a Scarf resolution if and
  only if $G$ is a gap-free forest.
\end{theorem}

\begin{proof} If $G$ is not a gap-free forest, then by
  \cref{lem.lem2B36} and \cref{lem.lem1B36}, $G$ is not Scarf. 
  
  Conversely, suppose that $G$ is a gap-free forest. Let $H$ be a connected component of $G$. By \cref{lem.lem2B36}, $H$ does not contain any induced
        $P_4$. Thus, $H$ is a tree of one of the following forms:

        \begin{center}
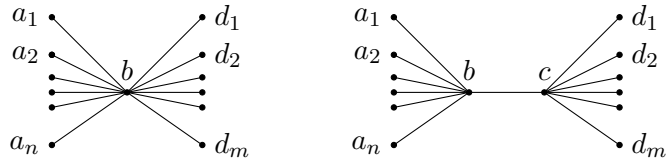
\begin{tabular}{ccc}	
	\begin{tikzpicture}
		\tikzstyle{point}=[inner sep=0pt]
		\node (a_1)[point,label=left:$a_1$] at (1,2.5) {};
		\node (b)[point,label=above:$b$] at (2,1.5) {};
		\node (a)[point] at (1,1.7) {};
		\node (a')[point] at (1,1.5) {};
		\node (a'')[point] at (1,1.3) {};
		\node (a_2)[point,label=left:$a_2$] at (1,2) {};
		\node (a_n)[point,label=left:$a_n$] at (1,.8) {};
		\node (d_1)[point,label=right:$d_1$] at (3,2.5) {};
		\node (d_2)[point,label=right:$d_2$] at (3,2) {};
		\node (d)[point] at (3,1.7) {};
		\node (d')[point] at (3,1.5) {};
		\node (d'')[point] at (3,1.3) {};
		\node (d_m)[point,label=right:$d_m$] at (3,.8) {};
		
		\draw (a_1.center) -- (b.center);
		\draw (a_2.center) -- (b.center);
		\draw (a_n.center) -- (b.center);
		\draw (a.center) -- (b.center);
		\draw (a'.center) -- (b.center);
		\draw (a''.center) -- (b.center);
			\draw (d_1.center) -- (b.center);
	\draw (d_2.center) -- (b.center);
	\draw (d_m.center) -- (b.center);
	\draw (d.center) -- (b.center);
	\draw (d'.center) -- (b.center);
	\draw (d''.center) -- (b.center);		
		\filldraw [black] (a.center) circle (1pt);
		\filldraw [black] (a'.center) circle (1pt);
		\filldraw [black] (a''.center) circle (1pt);
		\filldraw [black] (a_1.center) circle (1pt);
		\filldraw [black] (a_2.center) circle (1pt);
		\filldraw [black] (a_n.center) circle (1pt);
		\filldraw [black] (b.center) circle (1pt);
		\filldraw [black] (c.center) circle (1pt);
		\filldraw [black] (d_1.center) circle (1pt);
\filldraw [black] (d_2.center) circle (1pt);
\filldraw [black] (d_m.center) circle (1pt);
		\filldraw [black] (d.center) circle (1pt);
		\filldraw [black] (d'.center) circle (1pt);
		\filldraw [black] (d''.center) circle (1pt);
	\end{tikzpicture}
	&\quad \quad&
	\begin{tikzpicture}
		\tikzstyle{point}=[inner sep=0pt]
		\node (a_1)[point,label=left:$a_1$] at (1,2.5) {};
		\node (b)[point,label=above:$b$] at (2,1.5) {};
		\node (a)[point] at (1,1.7) {};
		\node (a')[point] at (1,1.5) {};
		\node (a'')[point] at (1,1.3) {};
		\node (a_2)[point,label=left:$a_2$] at (1,2) {};
		\node (a_n)[point,label=left:$a_n$] at (1,.8) {};
		\node (c)[point,label=above:$c$] at (3,1.5) {};
		\node (d_1)[point,label=right:$d_1$] at (4,2.5) {};
		\node (d_2)[point,label=right:$d_2$] at (4,2) {};
		\node (d)[point] at (4,1.7) {};
		\node (d')[point] at (4,1.5) {};
		\node (d'')[point] at (4,1.3) {};
		\node (d_m)[point,label=right:$d_m$] at (4,.8) {};
		
		\draw (a_1.center) -- (b.center);
		\draw (a_2.center) -- (b.center);
		\draw (a_n.center) -- (b.center);
		\draw (a.center) -- (b.center);
		\draw (a'.center) -- (b.center);
		\draw (a''.center) -- (b.center);
			\draw (d_1.center) -- (c.center);
	\draw (d_2.center) -- (c.center);
	\draw (d_m.center) -- (c.center);
	\draw (d.center) -- (c.center);
	\draw (d'.center) -- (c.center);
	\draw (d''.center) -- (c.center);
	\draw (b.center) -- (c.center);
		
		\filldraw [black] (a.center) circle (1pt);
		\filldraw [black] (a'.center) circle (1pt);
		\filldraw [black] (a''.center) circle (1pt);
		\filldraw [black] (a_1.center) circle (1pt);
		\filldraw [black] (a_2.center) circle (1pt);
		\filldraw [black] (a_n.center) circle (1pt);
		\filldraw [black] (b.center) circle (1pt);
		\filldraw [black] (c.center) circle (1pt);
		\filldraw [black] (d_1.center) circle (1pt);
\filldraw [black] (d_2.center) circle (1pt);
\filldraw [black] (d_m.center) circle (1pt);
		\filldraw [black] (d.center) circle (1pt);
		\filldraw [black] (d'.center) circle (1pt);
		\filldraw [black] (d''.center) circle (1pt);
		
	\end{tikzpicture}
\end{tabular}
        \end{center}
where $m \geq 0$ and $n \geq 0$. 

By \cref{t:tree}, the Scarf
        complex of a tree of the first form is a simplex on $(n+m)$ vertices and the Scarf complex of a tree of the second form is two simplices of
        sizes $(n+1)$ and $(m+1)$ joined at the common vertex $bc$,
        respectively. Note that every induced subcomplex of a simplex is again a simplex. In the second form of $\Scarf(H)$, every label in the first simplex is divisible by $b$ and every label in the second is divisible by $c$. If $\bm$ is in the LCM lattice and $bc$ divides $\bm$, then $\Delta_{\bm}$ is again two (sub)simplices joined at a point labeled $bc$, which is collapsible to a point and and hence acyclic. If $bc$ does not divide $\bm$, then without loss of generality, $b$ does not divide $\bm$. In this case $\Delta_{\bm}$ is a subsimplex of the simplex where each vertex is divisible by $c$.  Thus in both of these cases, $\Scarf(H)$ restricted to any label is acyclic.
        
        It now follows from \cref{l:join} that the same is true for $\Scarf(G)$. Therefore, $\Scarf(G)$ supports the minimal free resolution of $I(G)$ by \cref{t:BPS}. Hence $G$ is Scarf.
\end{proof}

\begin{remark} It is known that the regularity of $I(G)$, where $G$ is a forest, is the same as its induced matching number plus 1 (cf. \cite[Theorem 2.18]{Zheng} and \cite[Corollary 3.11]{HVT}). Also, a connected gap-free graph has induced matching number 1. Thus, if $I(G)$ has a Scarf resolution then the regularity of $I(G)$ is equal to the number of connected components of $G$ plus 1.
\end{remark}

\section{The Scarf complex of powers of edge ideals}
\label{s:Scarf-I2}

This section addresses the second part of our \say{Beautiful
  Oberwolfach Theorem}, when the power $t$ is at least
$2$. Specifically, we shall characterize graphs for which powers of
their edge ideals are Scarf.

Our investigation begins with a lemma identifying edges that do not
appear in $\Scarf(I(G))^t$, for a graph $G$ and a positive integer
$t$.  Below, for a monomial $\m$ and a variable $x$, we denote by
$\deg_x(\m)$ the highest power of $x$ appearing in $\m$.

\begin{lemma}[{\bf Non-Scarf edges}]\label{l:non-scarf-edge}
  Let $G$ be a graph with edge ideal $I = I(G)$, and let $t$ a positive
  integer, $e,e' \in \Gens(I)$, and $\bar{\bm},\bar{\bm}' \in
        \Gens(I^{t-1})$ such that $$\bm=e \cdot \bar{\bm} \qand \bm'=e' \cdot \bar{\bm}'.$$ Then in either of the following cases $\{\bm, \bm'\}\notin \Scarf(I^t)$.

   \begin{enumerate}
   \item \label{i:1} If $e$ and $e'$ are
     two distinct edges of a triangle in $G$ and $\bar{\bm} =
     \bar{\bm}'$.
\item \label{i:2} If $e=ab$, $e'=cd$, $bc \in \Gens(I)$, and
  \begin{align}
    & \deg_b(\bm') < \deg_b(\bm) \qand bc \cdot \bar{\bm}' \not= \bm, \qor \label{e:deg1} \\
    & \deg_c(\bm) <  \deg_c(\bm') \qand bc \cdot \bar{\bm} \not= \bm'. \label{e:deg2}
\end{align}
       \end{enumerate}
\end{lemma}

\begin{proof}

  \eqref{i:1} Suppose that $e=ab$ and $e'=bc$. Since $e$ and $e'$ are edges of
  a triangle in $G$, the edge $ac$ belongs to $G$. We write
  $\bm=ab\cdot \bar{\bm}$ and $\bm'=bc \cdot \bar{\bm}$. Then we
  have $$\lcm(\bm,\bm')=abc\cdot \bar{\bm}=\lcm(\bm,\bm',ac \cdot \bar{\bm})$$
  which shows that $\{\bm,\bm'\}$ shares a label with another face of
  $\Taylor(I^t)$, and is therefore not a face of $\Scarf(I^t)$.

  \eqref{i:2} Note that
  $$c \cdot \bar{\bm}' \mid \bm' \mid \lcm(\bm,\bm')
  \qand
  b \cdot \bar{\bm} \mid \bm \mid \lcm(\bm,\bm') .$$
  Now by \eqref{e:deg1} and \eqref{e:deg2} we have
  $$\deg_b(bc \cdot \bar{\bm}') = \deg_b(\bm')+1 \leq \deg_b(\bm)
  \qor  \deg_c(bc \cdot \bar{\bm}) = \deg_c(\bm)+1 \leq \deg_c(\bm').$$
  Therefore
  $$bc \cdot \bar{\bm}' \mid \lcm(\bm,\bm') \qor
  bc \cdot \bar{\bm} \mid \lcm(\bm,\bm').$$
  In these cases, we also have that $bc \cdot \bar{\bm}$ or $bc \cdot \bar{\bm}'$, respectively, does not belong to $\{\bm, \bm'\}$. Hence, $\{\bm,\bm'\} \notin \Scarf(I^t)$.
\end{proof}

Our main result is established by explicitly constructing the Scarf complexes of $I(G)^t$ for special classes of graphs $G$, namely, triangles, paths, squares and claws. This is done in the next theorem.

\begin{theorem}[{\bf The Scarf complex of powers of some basic subgraphs}]\label{t:non-Scarf}
  Let $G$ be a graph with edge ideal $I=I(G)$, and let $t \in \NN$.
      \begin{enumerate}
      \item If $t\geq 1$ and $G$ is a triangle, then $\Scarf(I^t)$ is a set of  ${t+2 \choose 2}$ isolated  vertices.
      \item If $t \geq 2$ and $G$ is a path of length $3$, then $\Scarf(I^t)$
        is  the upper triangular portion of a $t \times t$ grid of squares as in \cref{f:powers-path}.
      \item If $t \geq 2$ and $G$ is a claw, then $\Scarf(I^t)$ is an cycle of length $3t$ together with ${{2+t}\choose{2}}-3t$ isolated vertices, as in \cref{f:powers-claw}.
      \item If $t \geq 1$ and $G$ is a square, then $\Scarf(I^t)$ is a $t \times t$ grid of squares as in \cref{f:powers-square}.
    \end{enumerate}
  In particular $I^t$ does not have a Scarf resolution in the above cases.
\end{theorem}

\begin{proof} We prove each item separately below.

                  \subsection*{\bf (1) Powers of Triangles}
	          Let $G$ be a triangle with edge ideal $I = (ab, bc,
                  ca)$.  Observe that any minimal generator of $I^t$ is of the form $(ab)^s(bc)^q(ca)^r$, where $s+q+r = t$. By setting $\alpha = s+r$, $\beta = s+q$ and $\gamma = q+r$ (or equivalently, $s = t- \gamma, q = t - \alpha$ and $r = t-\beta$), it can be seen that
                  $$I^t = (a^\alpha b^\beta c^\gamma \mid \alpha+\beta+\gamma = 2t \qand 0 \le \alpha, \beta, \gamma \le t).$$

                  Consider any two distinct minimal generators $m = a^\alpha b^\beta c^\gamma$ and $n = a^i b^j c^k$ of $I^t$. Without loss of generality, we may assume that $i > \alpha$ and $j < \beta$. Let $p = a^i b^\beta c^{2t-i-\beta}$. Note that $i+\beta > \alpha+\beta = 2t-\gamma \ge t$. Then, $p$ is a minimal generator of $I^t$ that is not equal to $m$ or $n$, and
                  $$\lcm(m,n,p) = \lcm(m,n).$$
                  Thus, $\{m,n\} \notin \Scarf(I^t).$ We have shown that the 1-skeleton of $\Scarf(I^t)$ has no edges. Hence, $\Scarf(I^t)$ consists of isolated vertices. Using the standard combinatorial formula for counting with repetition, it can be seen that $\Scarf(I^t)$ has exactly ${3+t-1 \choose t} = {t+2 \choose 2}$ vertices.

\subsection*{\bf (2) Powers of Paths:}
Let $G$ be a path of length $3$ with edge ideal $I = (ab, bc, cd)$.
Then for $t \ge 2$,
	$$I^t = (a^ib^{t-k}c^{t-i}d^k \mid 0 \leq i,k \leq t \qand
i+k\leq t).$$ We shall first examine when the edge connecting two
vertices with distinct labels, $m = a^ib^{t-k}c^{t-i}d^k$ and
$n=a^\alpha b^{t-\beta} c^{t-\alpha}d^\beta$, in $\Taylor(I^t)$ is an
edge of $\Scarf(I^t)$. Since we cannot have both $i = \alpha$ and $k = \beta$, without loss of generality, we may assume that $i> \alpha$. In this case, $i > 0$ and $m = ab \cdot \bar{m}$, where $\bar{m} \in I^{t-1}$.

If $\beta \not= 0$, then $n = cd \cdot \bar{n}$, where $\bar{n} \in I^{t-1}$. Also, $\deg_c(m) = t-i < t-\alpha = \deg_c(n)$. Thus, if $bc \cdot \bar{m} = a^{i-1}b^{t-k} c^{t-i+1} d^k \not= n$ (i.e., if $\alpha \not= i-1$ or $k \not= \beta$) then, by \cref{l:non-scarf-edge}, $\{m,n\} \notin \Scarf(I^t)$.

Suppose that $\beta = 0$. Consider $p = \frac{m}{ab} \cdot bc = a^{i-1}b^{t-k}c^{t-i+1}d^k$. If $p \not= n$ (i.e., if $\alpha \not= i-1$ or $k \not= 0$) then we have $\lcm(m,n,p) = \lcm(m,n)$, so $\{m,n\} \notin \Scarf(I^t)$.
	
It remains to consider the case where $\alpha = i-1$ and $k = \beta$
We claim that, in this case, $\{m,n\} \in \Scarf(I^t)$.
	Indeed, suppose that there exists another vertex whose label $q = a^xb^{t-y}c^{t-x}d^y$ divides $\lcm(m,n) = a^ib^{t-k}c^{t-i+1}d^k$. Then, $y \le k$ and $t-y \le t-k$. Thus, $y=k$. Also, $x \leq i$ and $t-x \leq t-\alpha = t-i+1$. Therefore, either $x=i$ or $x=i-1 = \alpha$. It then follows that either $q = m$ or $q = n$.
	
	By symmetry, the edge connecting vertices with labels $m$ and $n$ is in $\Scarf(I^t)$ if $\alpha = i$ and $\beta = k-1$.
	
	We have shown that the 1-skeleton of the Scarf complex
        $\Scarf(I^t)$ has the form in \cref{f:powers-path}.  Observe
        that there are no cliques of size larger than $2$ in this
        1-skeleton of $\Scarf(I^t)$. This forces the Scarf complex of
        $I^t$ to be exactly the same as its 1-skeleton, up to
        isolated vertices. By a standard counting argument, since $I$ has three generators, $I^t$ has at most ${{3+t-1}\choose{t}}={{2+t}\choose{2}}$ generators, which is precisely the number of vertices in \cref{f:powers-path}. Thus there are no isolated vertices and the Scarf complex is precisely the 1-skeleton.

        		\begin{figure}
        	\begin{tikzpicture}
        		\tikzstyle{point}=[inner sep=0pt]
        		\node (a)[point,label=left:{\tiny $a^tb^t$}] at (0,6) {};
        		\node (b)[point,label=above:{\tiny $a^{t{-}1}b^tc$}] at (2,6) {};
        		\node (c)[point,label=above:{\tiny $a^{t{-}2}b^tc^2$}] at (4,6) {};
        		\node (cc)[point,label=above:{\tiny $a^{t{-}3}b^tc^3$}] at (6,6) {};
        		\node (d)[point,label=above:{\tiny $ab^tc^{t{-}1}$}] at (8,6) {};
        		\node (e) [point,label=above:{\tiny $b^tc^t$}] at (10,6) {};
        		\node (f) [point,label=left:{\tiny $a^{t{-}1}b^{t{-}1}cd$}] at (2,5) {};
        		\node (g) [point,label=above:{\tiny $a^{t{-}2}b^{t{-}1}c^2d$}] at (4,5) {};
        		\node (gg) [point,label=above:{\tiny $a^{t{-}3}b^{t{-}1}c^3d$}] at (6,5) {};
        		\node (h) [point,label=above:{\tiny $ab^{t{-}1}c^{t{-}1}d$}] at (8,5) {};
        		\node (i) [point,label=above:{\tiny $b^{t{-}1}c^td$}] at (10,5) {};
        		\node (k) [point,label=below:{\tiny $a^{t{-}2}b^{t{-}2}c^2d^2$}] at (4,4) {};
        		\node (j) [point,label=below:{\tiny $a^{t{-}3}b^{t{-}2}c^3d^2$}] at (6,4) {};
        		\node (l) [point,label=below:{\tiny $ac^{t{-}1}d^t$}] at (8,2) {};
        		\node (m) [point,label=below:{\tiny $c^{t}d^{t{-}1}$}] at (10,2) {};
        		\node (n) [point,label=below:{\tiny $c^td^t$}] at (10,1) {};
        		\node (q) [point,label=below:{\tiny $abc^{t{-}1}d^{t{-}1}$}] at (8,3) {};
        		\node (r) [point,label=below:{\tiny $bc^{t{-}1}d^{t{-}1}$}] at (10,3) {};

        		\draw (a.center) -- (b.center);
        		\draw (b.center) -- (c.center);
        		\draw (c.center) -- (cc.center);
        		\draw (d.center) -- (e.center);
        		\draw (b.center) -- (f.center);
        		\draw (c.center) -- (g.center);
        		\draw (f.center) -- (g.center);
        		\draw (k.center) -- (g.center);
        		\draw (g.center) -- (gg.center);
        		\draw (h.center) -- (i.center);
        		\draw (d.center) -- (h.center);
        		\draw (e.center) -- (i.center);
        		\draw (j.center) -- (k.center);
        		\draw (l.center) -- (m.center);
        		\draw (m.center) -- (n.center);
        		\draw (cc.center) --(gg.center);
        		\draw (gg.center) -- (j.center);
        		\draw (q.center) -- (l.center);
        		\draw (m.center) -- (r.center);
        		\draw (q.center) -- (r.center);
        		
        		\filldraw [black] (a.center) circle (1pt);
        		\filldraw [black] (b.center) circle (1pt);
        		\filldraw [black] (c.center) circle (1pt);
        		\filldraw [black] (cc.center) circle (1pt);
        		\filldraw [black] (d.center) circle (1pt);
        		\filldraw [black] (e.center) circle (1pt);
        		\filldraw [black] (f.center) circle (1pt);
        		\filldraw [black] (g.center) circle (1pt);
        		\filldraw [black] (gg.center) circle (1pt);
        		\filldraw [black] (h.center) circle (1pt);
        		\filldraw [black] (i.center) circle (1pt);
        		\filldraw [black] (j.center) circle (1pt);
        		\filldraw [black] (k.center) circle (1pt);
        		\filldraw [black] (l.center) circle (1pt);
        		\filldraw [black] (m.center) circle (1pt);
        		\filldraw [black] (n.center) circle (1pt);
        		\filldraw [black] (q.center) circle (1pt);
        		\filldraw [black] (r.center) circle (1pt);
        		
        		\node [point,label=left:{\tiny $\ldots$}] at (7.6,6) {};
        		\node [point,label=left:{\tiny $\ldots$}] at (7.6,5) {};
        		\node [point,label=left:{\tiny $\ldots$}] at (7.6,4.4) {};
        		\node [point,label=left:{\tiny $\ddots$}] at (7.2,3.5) {};
        		\node [point,label=left:{\tiny $\vdots$}] at (8.2,3.5) {};
        		\node [point,label=left:{\tiny $\vdots$}] at (10.2,3.5) {};
        		\node [point,label=left:{\tiny $\vdots$}] at (10.2,4.5) {};
        	\end{tikzpicture}
        	\caption{The Scarf complex of powers of a path $I=(ab,bc,cd)$}\label{f:powers-path}
        \end{figure}

\subsection*{\bf (3) Powers of Claws:}
	Let $G$ be a claw with edge ideal $I = (ab, ac, ad)$.  It can
        be seen that, for $t \in \NN$,
	$$I^t = (a^tb^ic^jd^k \mid 0 \le i,j,k \le t \qand i+j+k = t).$$
	
	Similar to previous cases, we start by examining when the
        edge connecting two vertices, with distinct labels $m =
        a^tb^ic^jd^k$ and $n = a^tb^\alpha c^\beta d^\gamma$, in
        $\Taylor(I^t)$ is an edge inside $\Scarf(I^t)$.
	
	Consider first the case where $|i - \alpha| \ge 2$. Without loss of generality, we may assume that $i \ge \alpha+2$. Since $i+j+k = \alpha+\beta+\gamma = t$, we must have either $\beta > j$ or $\gamma > k$. Suppose that $\beta > j$. Set $p = a^tb^{\alpha+1}c^{\beta-1}d^\gamma$. Clearly, $p \not= m,n$ and $\lcm(m,n,p) = \lcm(m,n) = a^tb^ic^\beta d^{\max\{k, \beta\}}$. Thus, the edge connecting vertices with labels $m$ and $n$ is not in $\Scarf(I^t)$.
	
	A similar argument works for the cases where $|j-\beta| \ge 2$ or $|k-\gamma| \ge 2$. It remains to consider the case where $|i-\alpha|, |j-\beta|, |k-\gamma| \le 1$. Since $\alpha+\beta+\gamma = i+j+k = t$, elements in exactly one of these pairs must be the same. Without loss of generality, we may assume that $k= \gamma$, $i = \alpha+1$ and $j = \beta-1$.
	
	If $k > 0$ then set $p = a^t b^i c^{j+1} d^{k-1} = \frac{m}{ad}\cdot ac$. Observe that $p \not= m,n$ and $\lcm(m,n,p) = \lcm(m,n) = a^t b^ic^\beta d^k$. This, again, implies that the edge connecting vertices with labels $m$ and $n$ is not in $\Scarf(I^t)$.
	
	If $k=\gamma = 0$ then $m = a^tb^ic^j$ and $n = a^tb^{i-1}c^{j+1}$, with $i+j = t$. We claim that, in this case, the edge connecting vertices with labels $m$ and $n$ is in $\Scarf(I^t)$. Indeed, suppose that $q = a^t b^xc^yd^z$, where $x+y+z = t$, is the label of another vertex that divides $\lcm(m,n) = a^t b^ic^{j+1}$. Then, $z = 0$, $x \le i$ and $y \le j+1$. Since $i+j=x+y=t$, this implies that either $x=i-1$ and $y=j+1$ or $x=i$ and $y=j$, i.e., either $q = m$ or $q = n$.
	
	We have shown that the 1-skeleton of $\Scarf(I^t)$ is as depicted in \cref{f:powers-claw}, which is a $3t$ cycle. Since $t \ge 2$, there are no cliques of size larger than $2$ in the 1-skeleton. It then follows that $\Scarf(I^t)$ is exactly the same as its 1-skeleton, a $3t$-cycle, together with isolated vertices. Using a standard counting argument, there are at most ${{3+t-1}\choose{t}}={{t+2}\choose{t}}$ generators of $I^t$. Since each of $b, c, d$ appear in precisely one generator of $I$,  there are exactly ${{t+2}\choose{2}}$ generators, of which all but $3t$ are isolated.

	\begin{figure}
	\begin{tikzpicture}
		\tikzstyle{point}=[inner sep=0pt]
		\node (1)[point,label=above:{$a^tb^t$}] at (3,3) {};
		\node (2)[point,label=above:{$a^tb^{t{-}1}c$}] at (4.5,3) {};
		\node (3)[point,label=right:{$a^tb^{t{-}2}c^2$}] at (5.9,2.5) {};
		\node (a)[point] at (6.5,1.2) {};
		\node (a')[point] at (6.5,1) {};
		\node (a'')[point] at (6.5,0.8) {};
		\node (4)[point,label=right:{$a^tc^t$}] at (5.9,-0.5) {};
		\node (5)[point,label=below:{$a^tc^{t{-}1}d$}] at (4.5,-1) {};
		\node (6)[point,label=below:{$a^td^t$}] at (3,-1) {};
		\node (7)[point,label=left:{$a^tb^{t{-}1}d$}] at (1.5,2.5) {};
		\node (8)[point] at (0.9,1.2) {};
		\node (10)[point] at (0.9,1) {};
		\node (9)[point] at (0.9, 0.8) {};
		\node (11)[point,label=left:{$a^tbd^{t{-}1}$}] at (1.5,-0.5) {};
		\node (b)[point] at (3.6,-1) {};
		\node (b')[point] at (3.8,-1) {};
		\node (b'')[point] at (4,-1) {};
		
		\draw (1.center) -- (2.center);
		\draw (2.center) -- (3.center);
		\draw (3.center) -- (a.center);
		\draw (4.center) -- (5.center);
		
		\draw (6.center) -- (11.center);
		\draw (7.center) -- (1.center);
		\draw (7.center) -- (8.center);
		\draw (9.center) -- (11.center);
		\draw (4.center) -- (a''.center);

		\filldraw [black] (1.center) circle (1pt);
		\filldraw [black] (2.center) circle (1pt);
		\filldraw [black] (3.center) circle (1pt);
		\filldraw [black] (a.center) circle (1pt);
		\filldraw [black] (a'.center) circle (1pt);
		\filldraw [black] (a''.center) circle (1pt);
		\filldraw [black] (4.center) circle (1pt);
		\filldraw [black] (5.center) circle (1pt);
		\filldraw [black] (6.center) circle (1pt);
		\filldraw [black] (7.center) circle (1pt);
		\filldraw [black] (8.center) circle (1pt);
		\filldraw [black] (9.center) circle (1pt);
		\filldraw [black] (10.center) circle (1pt);
		\filldraw [black] (11.center) circle (1pt);
		\filldraw [black] (b.center) circle (1pt);
		\filldraw [black] (b'.center) circle (1pt);
		\filldraw [black] (b''.center) circle (1pt);
	\end{tikzpicture}
	\caption{$1$-Skeleton of the Scarf complex of powers of a claw $I=(ab,ac,ad)$}\label{f:powers-claw}
\end{figure}
	
\subsection*{\bf (4) Powers of Squares:}
	Let $G$ be a square with edge ideal $I = (ab, bc, cd,
        da)$. Observe that, for $t \in \NN$,
	$$I^t = (a^i b^j c^{t-i} d^{t-j} \mid 0\le i \le t \text{ and } 0 \le j \le t).$$
	As in previous cases, we start by examining when the edge connecting two vertices, with distinct labels $m = a^i b^j c^{t-i} d^{t-j}$ and $n = a^\alpha b^\beta c^{t-\alpha} d^{t-\gamma}$, of $\Taylor(I^t)$ remains an edge in $\Scarf(I^t)$.
	
	If $|i - \alpha| \ge 2$, then similar to what was done with the claw, by assuming that $i \ge \alpha+2$ and considering $p = a^{\alpha+1}b^\beta c^{t-\alpha+1} d^{t-\beta}$, we conclude that the edge connecting $m$ and $n$ is not in $\Scarf(I^t)$. The same argument works for the case where $|j - \beta| \ge 2$.
	
	Suppose that $|i-\alpha|, |j-\beta| \le 1$. Without loss of generality, we may assume that $i > \alpha$; that is, $\alpha = i-1$. If $|j-\beta| = 1$, then by considering $p = a^i b^\beta c^{t-i} d^{t-\beta}$, we again conclude that the edge between $m$ and $n$ is not in $\Scarf(I^t)$.
	
	It remains to consider the case where $\alpha = i-1$ and $\beta = j$ (and, by symmetry, when $\alpha = i$ and $\beta = j-1$). We claim that, in this case, the edge between $m$ and $n$ is an edge in $\Scarf(I^t)$. Indeed, suppose that $q = a^xb^yc^{t-x}d^{t-y}$ is another vertex of $\Taylor(I^t)$ that divides $\lcm(m,n) = a^ib^jc^{t-i+1}d^{t-j}$. Then, $x \le i$, $y \le j$, $t-x \le t-i+1$ and $t-y \le t-j$. This implies that $y=j$ and either $x=i$ or $x = i-1$. That is, either $q = m$ or $q = n$.
	
	We have shown that the $1$-skeleton of $\Scarf(I^t)$ is as
        depicted in \cref{f:powers-square}. As before, since there is
        no clique of size larger than $2$ in the $1$-skeleton of
        $\Scarf(I^t)$, $\Scarf(I^t)$ is exactly the same as its
        $1$-skeleton, possibly together with isolated vertices. Using the form of $I^t$ above, there are $t+1$ choices for both $i$ and $j$ that yield $(t+1)^2$ distinct monomial generators of $I^t$. Since there are $(t+1)^2$ vertices in \cref{f:powers-square}, the Scarf complex is precisely the $1$-skeleton, with no isolated vertices.

		\begin{figure}
	\begin{tikzpicture}
		\tikzstyle{point}=[inner sep=0pt]
		\node (a)[point,label=left:{\tiny $a^tb^t$}] at (0,5) {};
		\node (b)[point,label=above:{\tiny $a^{t{-}1}b^tc$}] at (2,5) {};
		\node (c)[point,label=above:{\tiny $a^{t{-}2}b^tc^2$}] at (4,5) {};
		\node (d)[point,label=above:{\tiny $ab^tc^{t{-}1}$}] at (6,5) {};
		\node (e) [point,label=right:{\tiny $b^tc^t$}] at (8,5) {};
		\node (2) [point,label=left:{\tiny $a^tb^{t{-}1}d$}] at (0,4) {};
		\node (f) [point,label=above:{\tiny $a^{t{-}1}b^{t{-}1}cd$}] at (2,4) {};
		\node (g) [point,label=above:{\tiny $a^{t{-}2}b^{t{-}1}c^2d$}] at (4,4) {};
		\node (h) [point,label=above:{\tiny $ab^{t{-}1}c^{t{-}1}d$}] at (6,4) {};
		\node (i) [point,label=right:{\tiny $b^{t{-}1}c^td$}] at (8,4) {};
		\node (3) [point,label=left:{\tiny $a^tbd^{t{-}1}$}] at (0,3) {};
		\node (j) [point,label=below:{\tiny $a^{t{-}1}cd^t$}] at (2,2) {};
		\node (k) [point,label=below:{\tiny $a^{t{-}2}c^2d^t$}] at (4,2) {};
		\node (l) [point,label=below:{\tiny $ac^{t{-}1}d^t$}] at (6,2) {};
		\node (4) [point,label=left:{\tiny $a^td^t$}] at (0,2) {};
		\node (m) [point,label=right:{\tiny $c^{t}d^{t}$}] at (8,2) {};
		
		\node (o) [point,label=below:{\tiny $a^{t{-}1}bcd^{t{-}1}$}] at (2,3) {};
		\node (p) [point,label=below:{\tiny $a^{t{-}2}bc^2d^{t{-}1}$}] at (4,3) {};
		\node (q) [point,label=below:{\tiny $abc^{t{-}1}d^{t{-}1}$}] at (6,3) {};
		\node (r) [point,label=right:{\tiny $bc^{t}d^{t{-}1}$}] at (8,3) {};

		\draw (a.center) -- (b.center);
		\draw (b.center) -- (c.center);
		\draw (d.center) -- (e.center);
		\draw (b.center) -- (f.center);
		\draw (c.center) -- (g.center);
		\draw (f.center) -- (g.center);
		\draw (h.center) -- (i.center);
		\draw (d.center) -- (h.center);
		\draw (e.center) -- (i.center);
		\draw (j.center) -- (k.center);
		\draw (l.center) -- (m.center);
		
		\draw (2.center) -- (a.center);
		\draw (2.center) -- (f.center);
		\draw (3.center) -- (4.center);
		\draw (3.center) -- (o.center);
		
		\draw (4.center) -- (j.center);
		\draw (r.center) -- (q.center);
		
		\draw (m.center) -- (r.center);
		\draw (l.center) -- (q.center);
		\draw (k.center) -- (p.center);
		\draw (j.center) -- (o.center);
		\draw (o.center) -- (p.center);
		
		\filldraw [black] (a.center) circle (1pt);
		\filldraw [black] (b.center) circle (1pt);
		\filldraw [black] (c.center) circle (1pt);
		\filldraw [black] (d.center) circle (1pt);
		\filldraw [black] (e.center) circle (1pt);
		\filldraw [black] (f.center) circle (1pt);
		\filldraw [black] (g.center) circle (1pt);
		\filldraw [black] (h.center) circle (1pt);
		\filldraw [black] (i.center) circle (1pt);
		\filldraw [black] (j.center) circle (1pt);
		\filldraw [black] (k.center) circle (1pt);
		\filldraw [black] (l.center) circle (1pt);
		\filldraw [black] (m.center) circle (1pt);
		
		\filldraw [black] (o.center) circle (1pt);
		\filldraw [black] (p.center) circle (1pt);
		\filldraw [black] (q.center) circle (1pt);
		\filldraw [black] (r.center) circle (1pt);
		\filldraw [black] (2.center) circle (1pt);
		\filldraw [black] (3.center) circle (1pt);
		\filldraw [black] (4.center) circle (1pt);
		
		\node [point,label=left:{\tiny $\ldots$}] at (5.6,5) {};
		\node [point,label=left:{\tiny $\ldots$}] at (5.6,4) {};
		\node [point,label=left:{\tiny $\vdots$}] at (2.2,3.5) {};
		\node [point,label=left:{\tiny $\vdots$}] at (4.2,3.5) {};
		\node [point,label=left:{\tiny $\vdots$}] at (6.2,3.5) {};
		\node [point,label=left:{\tiny $\vdots$}] at (8.2,3.5) {};
		\node [point,label=left:{\tiny $\vdots$}] at (0.2,3.5) {};
		\node [point,label=left:{\tiny $\ldots$}] at (5.6,3) {};
		\node [point,label=left:{\tiny $\ldots$}] at (5.6,2) {};
	\end{tikzpicture}
	\caption{The Scarf complex of powers of a square $I=(ab,bc,cd,da)$}\label{f:powers-square}
\end{figure}

\smallskip

Finally, to complete the proof of the theorem we consider the monomial $\m=a^tb^tc^t$ in Case~(1) and $\m=
a^tb^tc^td^t$ in cases (2) - (4).  Then $\m \in \LCM(I^t)$ and
$\Scarf(I^t)_{\m}$ has nontrivial homology in dimension $0$ in
Case~(1) and in dimension $1$ in all other cases. Hence, by \cref{t:BPS},
$I^t$ does not have a Scarf resolution in any of the cases (1) - (4).
\end{proof}

The main results of our paper are finally summarized in the following theorem.

\begin{theorem}[{\bf The \say{Beautiful Oberwolfach Theorem}}]\label{t:beautiful}  
Let $G$ be a graph with edge ideal $I = I(G)$. 
  \begin{enumerate}
  \item $I$ has a minimal free resolution
  supported on its Scarf complex if and only if $G$ is a gap-free forest.
  \end{enumerate}
   If $G$ is connected and $t>1$, then
  \begin{enumerate}[resume]
  \item $I^t$ has a minimal free resolution
  supported on its Scarf complex if and only if $G$ is an isolated vertex, an edge, or a path of length $2$.
  \end{enumerate}
\end{theorem}

  \begin{proof} The case $t=1$ is in \cref{t:B36}. Suppose that $t \geq 2$ and
    $G$ is not one of the the graphs listed above. Then $G$ has an
    induced subgraph $G'$ which is a triangle, a path of length $3$, a
    square or a claw with three edges, as depicted below.

            \begin{center}
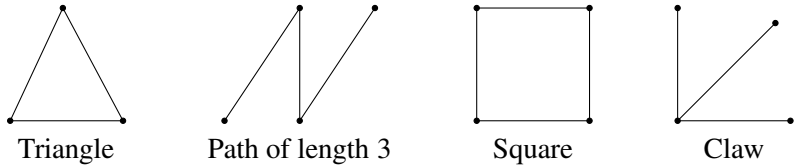
    	\begin{tabular}{ccccccc}	
    		\begin{tikzpicture}
    			\tikzstyle{point}=[inner sep=0pt]
    			\node (a)[point] at (0.7,1.5) {};
    			\node (b)[point] at (0,0) {};
    			\node (c)[point] at (1.5,0) {};
    			
    			\draw (a.center) -- (b.center);
    			\draw (b.center) -- (c.center);
    			\draw (c.center) -- (a.center);
    			
    			\filldraw [black] (a.center) circle (1pt);
    			\filldraw [black] (b.center) circle (1pt);
    			\filldraw [black] (c.center) circle (1pt);
    		\end{tikzpicture}
    		&\quad \quad&
    		\begin{tikzpicture}
    			\tikzstyle{point}=[inner sep=0pt]
    			\node (a)[point] at (0,0) {};
    			\node (b)[point] at (1,1.5) {};
    			\node (c)[point] at (1,0) {};
    			\node (d)[point] at (2,1.5) {};

    			\draw (a.center) -- (b.center);
    			\draw (b.center) -- (c.center);
    			\draw (c.center) -- (d.center);
    			
    			\filldraw [black] (a.center) circle (1pt);
    			\filldraw [black] (b.center) circle (1pt);
    			\filldraw [black] (c.center) circle (1pt);
    			\filldraw [black] (d.center) circle (1pt);   			
    		\end{tikzpicture}
    	   		&\quad \quad&
    	\begin{tikzpicture}
    		\tikzstyle{point}=[inner sep=0pt]
    		\node (a)[point] at (0,0) {};
    		\node (b)[point] at (0,1.5) {};
    		\node (d)[point] at (1.5,0) {};
    		\node (c)[point] at (1.5,1.5) {};
    		
    		\draw (a.center) -- (b.center);
    		\draw (b.center) -- (c.center);
    		\draw (c.center) -- (d.center);
    		\draw (a.center) -- (d.center);
    		
    		\filldraw [black] (a.center) circle (1pt);
    		\filldraw [black] (b.center) circle (1pt);
    		\filldraw [black] (c.center) circle (1pt);
    		\filldraw [black] (d.center) circle (1pt);   			
    	\end{tikzpicture}
       		&\quad \quad&
  \begin{tikzpicture}
    	\tikzstyle{point}=[inner sep=0pt]
    	\node (a)[point] at (0,0) {};
    	\node (b)[point] at (0,1.5) {};
    	\node (c)[point] at (1.3, 1.3) {};
    	\node (d)[point] at (1.5,0) {};
    	
    	\draw (a.center) -- (b.center);
    	\draw (a.center) -- (c.center);
    	\draw (a.center) -- (d.center);
    	
    	\filldraw [black] (a.center) circle (1pt);
    	\filldraw [black] (b.center) circle (1pt);
    	\filldraw [black] (c.center) circle (1pt);
    	\filldraw [black] (d.center) circle (1pt);   			
    \end{tikzpicture} \\
\text{Triangle} & & \text{Path of length 3} & & \text{Square} & & \text{Claw}
    	\end{tabular}
    \end{center}

 \noindent   It follows from  \cref{t:non-Scarf} that, in each of these
    cases, there is a monomial $u \in \LCM(I(G')^t$ where
    $\Scarf(I(G')^t)_{u}$ has nontrivial homology. \cref{l:induced}
      now implies that $u \in \LCM(I^t)$ and $\Scarf(I^t)_{u}$
      has nontrivial homology. This, together with \cref{t:BPS}, implies that $\Scarf(I^t)$ does not support a free resolution of $I^t$. 

    On the other hand, if $G$ is an isolated vertex ($I = (0)$) or $G$ is an edge ($I = (uv)$), then $I^t$ has a minimal free resolution supported by its Scarf complex in a trivial way. If $G$ is a path of length 2, then $G$ is an induced subgraph of the path of length 3. By \Cref{l:induced}, it follows that the Scarf complex of $I^t$ is of the form of the top row of the complex in \Cref{f:powers-path}. Hence, applying \Cref{t:BPS}, we conclude that $I^t$ admits a minimal free resolution supported by its Scarf complex.
  \end{proof}

  \begin{figure}[H]
	\begin{tikzpicture}
		\tikzstyle{point}=[inner sep=0pt]
		\node (a)[point,label=left:{\tiny $x^t$}] at (4,6) {};
		\node (b)[point,label=left:{\tiny $x^{t-1}ab$}] at (3,5) {};
		\node (c)[point,label=right:{\tiny $x^{t-1}bc$}] at (5,5) {};
		\node (d) [point,label=left:{\tiny $x^{t-2}a^2b^2$}] at (2,4) {};
        \node (e) [point,label=above:{\tiny $x^{t-2}ab^2c$}] at (4,4) {};
		\node (f) [point,label=right:{\tiny $x^{t-2}b^2c^2$}] at (6,4) {};
		\node (g) [point,label=left:{\tiny $xa^{t-1}b^{t-1}$}] at (1,3) {};
		\node (h) [point,label=below:{\tiny $xa^{t-2}b^{t-1}c$}] at (2,3) {};
        \node (o) [point,label=above:{\tiny $xa^{t-3}b^{t-1}c^2$}] at (3,3) {};
		\node (i) [point,label=below:{\tiny $xab^{t-1}c^{t-2}$}] at (6,3) {};
		\node (j) [point,label=right:{\tiny $xb^{t-1}c^{t-1}$}] at (7,3) {};
		\node (k) [point,label=left:{\tiny $a^tb^t$}] at (0,2) {};
		\node (l) [point,label=below:{\tiny $a^{t-1}b^tc$}] at (1,2) {};
        \node (p) [point,label=above:{\tiny $a^{t-2}b^tc^2$}] at (2,2) {};
        \node (q) [point,label=below:{\tiny $a^{t-3}b^tc^3$}] at (3,2) {};
		\node (m) [point,label=below:{\tiny $ab^tc^{t-1}$}] at (7,2) {};
		\node (n) [point,label=right:{\tiny $b^tc^t$}] at (8,2) {};

		\draw (a.center) -- (b.center);
		\draw (b.center) -- (c.center);
        \draw (c.center) -- (a.center);
		\draw (d.center) -- (e.center);
		\draw (e.center) -- (f.center);
        \draw (b.center) -- (d.center);
		\draw (c.center) -- (f.center);
		\draw (g.center) -- (h.center);
        \draw (o.center) -- (h.center);
		\draw (i.center) -- (j.center);
		\draw (k.center) -- (l.center);
        \draw (l.center) -- (p.center);
        \draw (p.center) -- (q.center);
		\draw (m.center) -- (n.center);
		\draw (g.center) -- (k.center);
		\draw (j.center) -- (n.center);

		\filldraw [black] (a.center) circle (1pt);
		\filldraw [black] (b.center) circle (1pt);
		\filldraw [black] (c.center) circle (1pt);
		\filldraw [black] (d.center) circle (1pt);
		\filldraw [black] (e.center) circle (1pt);
		\filldraw [black] (f.center) circle (1pt);
		\filldraw [black] (g.center) circle (1pt);
		\filldraw [black] (h.center) circle (1pt);
		\filldraw [black] (i.center) circle (1pt);
		\filldraw [black] (j.center) circle (1pt);
		\filldraw [black] (k.center) circle (1pt);
		\filldraw [black] (l.center) circle (1pt);
		\filldraw [black] (m.center) circle (1pt);
		\filldraw [black] (n.center) circle (1pt);
        \filldraw [black] (o.center) circle (1pt);
        \filldraw [black] (p.center) circle (1pt);
        \filldraw [black] (q.center) circle (1pt);

        \draw  [fill=gray!20] (a.center) -- (b.center) -- (c.center) -- cycle;
        
		\node [point,label=left:{\tiny $\iddots$}] at (1.9,3.6) {};
		\node [point,label=left:{\tiny $\vdots$}] at (4.4,3.6) {};
		\node [point,label=left:{\tiny $\ddots$}] at (6.9,3.6) {};
		\node [point,label=left:{\tiny $\ldots$}] at (5.6,3) {};
        \node [point,label=left:{\tiny $\ldots$}] at (4,3) {};
		\node [point,label=left:{\tiny $\ldots$}] at (4.6,2) {};
        \node [point,label=left:{\tiny $\ldots$}] at (6.4,2) {};
	\end{tikzpicture}
	\caption{The Scarf complex of powers of a disconnected graph with a path of length $2$ and an isolated vertex $I=(ab,bc,x)$}\label{f:powers-disconnected-path}
\end{figure}
  \begin{remark}
      Let $G$ be a disconnected graph and $t > 1$. A similar analysis to that of \Cref{t:non-Scarf} shows that:
      \begin{enumerate}
          \item If $G$ has a path $a -b -c$ of length 2 in at least one of the connected component then $\Scarf(I^t)$ contains a triangle together with a striated pyramid as in \Cref{f:powers-disconnected-path} as an induced subcomplex (where $x$ is a vertex in another connected component of $G$).
          \item If $G$ has at least 3 connected components then $\Scarf(I^t)$ contains an induced cycle of length $3t$ as in \Cref{f:powers-claw-2} as an induced subcomplex (where $u,x,z$ are vertices in different connected components of $G$, and $v,y,w$ are vertices connected to $u, x, z$, respectively, if such edges exist or 1 otherwise).
      \end{enumerate}
      Both of these cases result in induced subcomplexes with nontrivial homology, violating the condition in \Cref{t:BPS}.
     Hence, $I^t$ admits a minimal free resolution supported by its Scarf complex if and only if $G$ has exactly 2 connected components and each of these components is either an isolated vertex or an edge. We leave the detailed arguments to the interested reader.
  \end{remark}

\smallskip

\begin{figure}
	\begin{tikzpicture}
		\tikzstyle{point}=[inner sep=0pt]
		\node (1)[point,label=above:{\tiny $(uv)^t$}] at (3,3) {};
		\node (2)[point,label=above right:{\tiny $(uv)^{t{-}1}(xy)$}] at (4.5,3) {};
		\node (3)[point,label=right:{\tiny $(uv)^{t{-}2}(xy)^2$}] at (5.9,2.5) {};
		\node (a)[point] at (6.5,1.2) {};
		\node (a')[point] at (6.5,1) {};
		\node (a'')[point] at (6.5,0.8) {};
		\node (4)[point,label=right:{\tiny $(xy)^t$}] at (5.9,-0.5) {};
		\node (5)[point,label=below right:{\tiny $(xy)^{t{-}1}(zw)$}] at (4.5,-1) {};
		\node (6)[point,label=below:{\tiny $(zw)^t$}] at (3,-1) {};
		\node (7)[point,label=left:{\tiny $(uv)^{t{-}1}(zw)$}] at (1.5,2.5) {};
		\node (8)[point] at (0.9,1.2) {};
		\node (10)[point] at (0.9,1) {};
		\node (9)[point] at (0.9, 0.8) {};
		\node (11)[point,label=left:{\tiny $(uv)(zw)^{t{-}1}$}] at (1.5,-0.5) {};
		\node (b)[point] at (3.6,-1) {};
		\node (b')[point] at (3.8,-1) {};
		\node (b'')[point] at (4,-1) {};
		
		\draw (1.center) -- (2.center);
		\draw (2.center) -- (3.center);
		\draw (3.center) -- (a.center);
		\draw (4.center) -- (5.center);
		
		\draw (6.center) -- (11.center);
		\draw (7.center) -- (1.center);
		\draw (7.center) -- (8.center);
		\draw (9.center) -- (11.center);
		\draw (4.center) -- (a''.center);

		\filldraw [black] (1.center) circle (1pt);
		\filldraw [black] (2.center) circle (1pt);
		\filldraw [black] (3.center) circle (1pt);
		\filldraw [black] (a.center) circle (1pt);
		\filldraw [black] (a'.center) circle (1pt);
		\filldraw [black] (a''.center) circle (1pt);
		\filldraw [black] (4.center) circle (1pt);
		\filldraw [black] (5.center) circle (1pt);
		\filldraw [black] (6.center) circle (1pt);
		\filldraw [black] (7.center) circle (1pt);
		\filldraw [black] (8.center) circle (1pt);
		\filldraw [black] (9.center) circle (1pt);
		\filldraw [black] (10.center) circle (1pt);
		\filldraw [black] (11.center) circle (1pt);
		\filldraw [black] (b.center) circle (1pt);
		\filldraw [black] (b'.center) circle (1pt);
		\filldraw [black] (b''.center) circle (1pt);
	\end{tikzpicture}
	\caption{$1$-Skeleton of the Scarf complex of powers of a disconnected graph with three edges $I=(uv, xy, zw)$}\label{f:powers-claw-2}
\end{figure}

\subsubsection*{Acknowledgements}

The research for this paper was initiated during the authors' two week
stay at the Mathematisches Forschungsinstitut in Oberwolfach, as
Oberwolfach Research Fellows, and continued while three of the authors
were visiting the Vietnam Institute of Advanced Study in Mathematics
(VIASM).  We would like to thank MFO and VIASM for their warm
hospitality, and providing optimal research environments.

Our work was aided by computations using the computer
algebra software Macaulay2~\cite{M2}. We are also grateful to the referees for their helpful comments.

We would also like to acknowledge other research funding that was used
for this project. Sara Faridi was supported by NSERC Discovery Grant 2023-05929. T\`ai H\`a was supported by the Louisiana Board of Regents
and a Simons Foundation Grant.  Takayuki Hibi was partially supported
by JSPS KAKENHI 19H00637.


\end{document}